\documentclass[10pt,reqno]{amsart}

\usepackage{amsmath}
\usepackage{amsfonts}
\usepackage{amssymb}
\usepackage{amsthm}
\usepackage{mathrsfs}
\usepackage{latexsym}
\usepackage{cite} 

\usepackage{esint}

\numberwithin{equation}{section}

\usepackage{graphicx,subfigure}

\usepackage{enumitem}

\usepackage{ifthen} % pacchetto necessario per un minimo di
\provideboolean{shownotes} % dichiariamo una variabile booleana
\setboolean{shownotes}{true} % la definiamo come "vera" o "falsa"
\newcommand{\margnote}[1]{
\ifthenelse{\boolean{shownotes}}%
{\marginpar{\raggedright\tiny\texttt{#1}}}%
{}%
}

\newcommand{\hole}[1]{
\ifthenelse{\boolean{shownotes}}%
{\begin{center} \fbox{ \rule {.25cm}{0cm}
\rule[-.1cm]{0cm}{.4cm} \parbox{.85\textwidth}{\begin{center}
\texttt{#1}\end{center}} \rule {.25cm}{0cm}}\end{center}}
{}
}

\usepackage[colorlinks=true,urlcolor=blue,
citecolor=red,linkcolor=blue,linktocpage,pdfpagelabels,
bookmarksnumbered,bookmarksopen]{hyperref}

\theoremstyle{plain}

\newtheorem{lemma}{Lemma}[section]
\newtheorem{theorem}[lemma]{Theorem}

\newtheorem{corollary}[lemma]{Corollary}

\theoremstyle{definition}
\newtheorem{remark}[lemma]{Remark}
\newtheorem{definition}[lemma]{Definition}

\theoremstyle{remark}

\newcommand{\Id}{\mathrm{Id}}

\newcommand{\A}{\mathbb{A}}

\newcommand{\R}{\mathbb{R}}
\newcommand{\C}{\mathbb{C}}

\newcommand{\N}{\mathbb{N}}
\newcommand{\bbS}{\mathbb{S}}

\newcommand{\cO}{{\mathcal{O}}}
\newcommand{\cT}{{\mathcal{T}}}

\newcommand{\cM}{{\mathcal{M}}}
\newcommand{\cL}{{\mathcal{L}}}
\newcommand{\cD}{{\mathcal{D}}}

\newcommand{\cS}{{\mathcal{S}}}

\newcommand{\cA}{{\mathcal{A}}}

\newcommand{\cZ}{{\mathcal{Z}}}
\newcommand{\V}{\mathcal{V}}
\newcommand{\cV}{\mathcal{V}}
\newcommand{\cC}{{\mathcal{C}}}

\newcommand{\vep}{\varepsilon}

\renewcommand{\Re}{\mathrm{Re}\,}

\newcommand{\ep}{\epsilon}

\newcommand{\ess}{\sigma_\mathrm{\tiny{ess}}}
\newcommand{\ptsp}{\sigma_\mathrm{\tiny{pt}}}

\newcommand{\Ldper}{L^2_\mathrm{\tiny{per}}}
\newcommand{\Huper}{H^1_\mathrm{\tiny{per}}}
\newcommand{\Hdper}{H^2_\mathrm{\tiny{per}}}

\newcommand{\Hsper}{H^s_\mathrm{\tiny{per}}}
\newcommand{\Hsdper}{H^{s+2}_\mathrm{\tiny{per}}}
\newcommand{\Hper}{H_\mathrm{\tiny{per}}}

\begin{document}

\title[Orbital instability of periodic waves for viscous balance laws]{Orbital instability of periodic waves for scalar viscous balance laws}

\author[E. \'{A}lvarez]{Enrique \'{A}lvarez}
 
\address{{\rm (E. \'{A}lvarez)} Departamento de F\'{\i}sica Matem\'atica\\Instituto de 
Investigaciones en Matem\'aticas Aplicadas y en Sistemas\\Universidad Nacional Aut\'onoma de 
M\'exico\\ Circuito Escolar s/n, Ciudad Universitaria, C.P. 04510\\Cd. de M\'{e}xico (Mexico)}

\email{enrique.alvarez@ciencias.unam.mx}

\author[J. Angulo Pava]{Jaime Angulo Pava}

\address{{\rm (J. Angulo Pava)} Department of Mathematics\\IME-USP\\Rua do Mat\~ao 1010, Cidade Universit\'aria, CEP 05508-090,
 S\~ao Paulo, SP (Brazil)}

\email{angulo@ime.usp.br}

\author[R. G. Plaza]{Ram\'on G. Plaza}

\address{{\rm (R. G. Plaza)} Departamento de Matem\'aticas y Mec\'anica\\Instituto de 
Investigaciones en Matem\'aticas Aplicadas y en Sistemas\\Universidad Nacional Aut\'onoma de 
M\'exico\\ Circuito Escolar s/n, Ciudad Universitaria, C.P. 04510\\Cd. de M\'{e}xico (Mexico)}

\email{plaza@mym.iimas.unam.mx}

\begin{abstract}
The purpose of this paper is to prove that, for a large class of nonlinear evolution equations known as scalar viscous balance laws, the spectral (linear) instability condition of periodic traveling wave solutions implies their orbital (nonlinear) instability in appropriate periodic Sobolev spaces. The analysis is based on the well-posedness theory, the smoothness of the data-solution map, and an abstract result of instability of equilibria under nonlinear iterations. The resulting instability criterion is applied to two families of periodic waves. The first family consists of small amplitude waves with finite fundamental period which emerge from a local Hopf bifurcation around a critical value of the velocity. The second family comprises arbitrarily large period waves which arise from a homoclinic (global) bifurcation and tend to a limiting traveling pulse when their fundamental period tends to infinity. In the case of both families, the criterion is applied to conclude their orbital instability under the flow of the nonlinear viscous balance law in periodic Sobolev spaces with same period as the fundamental period of the wave.
\end{abstract}



\maketitle

\setcounter{tocdepth}{1}

% \tableofcontents

%%%%%%%%%%%%%%%%%\nocite{*}

\section{Introduction}

Scalar viscous balance laws in one space dimension are equations of the form
\begin{equation}
\label{VBL}
u_t + f(u)_x = u_{xx} + g(u),
\end{equation}
where $u = u(x,t) \in \R$ is a scalar unknown, $x \in \R$ and $t > 0$ denote the space and time variables, respectively, and $f = f(u)$ and $g = g(u)$ are nonlinear functions. Equation \eqref{VBL} describes the dynamics of a scalar quantity $u$ in a one dimensional domain, which is subject to three different mechanisms: the reaction $g(u)$ may describe production/consumption, chemical reactions or combustion, among other interactions; the density $u$ is (nonlinearly) transported with speed $f'(u)$; and the diffusion of $u$ is represented by the Laplace operator, $\partial_x^2$. In sum, scalar viscous balance laws constitute simplified models that combine diffusion (viscosity), convection and reaction into one single equation. For an abridged list of references on scalar viscous balance laws, see \cite{CroMas07,AlPl21,HaeSa06,Hae03}.

In these models, one of the most important mathematical solution types is the traveling wave. A \emph{spatially periodic} traveling wave solution to \eqref{VBL} has the form
\begin{equation}
\label{tws}
u(x,t) = \varphi(x-ct),
\end{equation}
where the constant $c \in \R$ is the speed of the wave and the profile function, $\varphi = \varphi(\xi)$, $\xi \in \R$, is a sufficiently smooth periodic function of its argument with fundamental period $L > 0$. 

In this paper, we are concerned with the stability of a periodic wave as a solution to the equation \eqref{VBL}. In general, stability of a specific traveling wave solution under small perturbations is a fundamental property for understanding the real-world dynamics of models of evolution type. Examples of  such models are the non-linear Schr\"odinger equation, the sine-Gordon equation and models of Korteweg-de Vries type, among others. The existence and stability theory of periodic waves has developed rapidly in recent years and it has drawn the attention of researchers from different areas of science, such as in fluid mechanics, optics, biology, or engineering, just to mention a few. New methods and theories have been developed using tools of non-linear analysis, bifurcation theory, spectral theory, as well as Fourier and harmonic analyses. The following (abridged) list of references may provide the reader with a panoramic idea of the recent evolution of the theory, \cite{AngAMS09, AngNat14, AngNat16, AnLN08, AngNat08, AngNat09,AlPl21,JMMP14,BJNRZ10}. In the particular case of scalar viscous balance laws, the analyses available in the literature have mainly focused on the stability of traveling fronts on the real line (see, e.g., \cite{Xing05, XuJJ16,WuXi05} and the references therein). In contrast, less attention has been paid to the stability of periodic waves.

%\hole{TO DO: (A) Begin.}
%
%Our work can be considered as an complement to the results in  \cite{AlPl21}. Thus, our study  shows that the property of spectral type instability in  \cite{AlPl21} for periodic traveling waves either small period amplitude or of large period approaching to the pulse solution to \eqref{VBL} (which will always exist in the established theory), it will be nonlinearly unstable by the induced flow of the model \eqref{VBL}. 
%
%Our general nonlinear instability results is based on a combination of the local well-posedness theory for \eqref{VBL}, implicit function theorem and the results in Henry {\it{et al.}} \cite{HPW82}. The Cauchy problem theory for \eqref{VBL} is based on the parabolic regularization method via Banach's fixed point theorem, although the proof is classical and without major problems, several estimates in the proof need to be established as they will be used in the smoothness result for the data-solution map.
%
%
%As an application of our results, we consider the general Fisher-KPP models and generalized Burgers-Fisher models (see \eqref{eqBF}-\eqref{MBFmodel}, \eqref{LogBLmodel}) just to mention a few. We consider that our results are completely new and it will help for the  understanding of the nonlinear dynamics for scalar viscous balance laws models in \eqref{VBL}.
%
%
%
%
%
%
%
%The paper is structured as follows.
%
%\hole{TO DO: (A) End.}

%\hole{TO DO: Possible text for (A):}

In this work, we are interested in establishing new results related to the dynamics of periodic traveling waves for models of the general type in \eqref{VBL}. The present analysis can be regarded as a complement to the recent study in \cite{AlPl21}, where the existence and spectral instability of periodic waves for equations of the form \eqref{VBL} were established. The natural question is whether this spectral information guarantees the instability of the waves under the nonlinear evolution. Hence, the purpose of this paper is to show that, if a periodic wave is spectrally unstable then it is also nonlinearly (orbitally) unstable under the flow of the evolution equation \eqref{VBL}. For such statement to be meaningful, it is crucial to specify the spaces under which the spectrum is calculated, and for which the well-posedness holds. Our instability criterion warrants the orbital instability of the manifold generated by any spectrally unstable periodic wave, under the flow of the nonlinear viscous balance law \eqref{VBL} in periodic Sobolev spaces with \emph{same period} as the fundamental period of the wave. 

The analysis is based on a combination of the local well-posedness theory for \eqref{VBL}, the implicit function theorem, and an (important) abstract result by Henry {\it{et al.}} \cite{HPW82}, which essentially determines the instability of a manifold of equilibria under iterations of a nonlinear map with unstable linearized spectrum. This general abstract theorem has been the basis of the nonlinear instability theory of periodic waves in other contexts, such as the KdV equation \cite{LopO02}, the critical KdV and NLS models \cite{AngNat09}, KdV systems \cite{AnLN08}, and for general dispersive models \cite{AngNat16}, just to mention a few. In order to apply the theorem by Henry \emph{et al.}, some essential elements are needed, such as a suitable well-posedness theory and the property that the data-solution map is of class $C^2$.

There exist several studies of well-posedness for parabolic equations of the form \eqref{VBL} available in the literature (see, e.g., \cite{AmannI-95,AmannII-19,LaSoUr68,Arm66}). For convenience of the reader, we present a detailed (yet concise) proof of local well-posedness of the Cauchy problem for equations of the form \eqref{VBL} in periodic Sobolev spaces of distributions (in the spirit of the analysis of Iorio and Iorio \cite{IoIo01} for nonlinear equations). Even though our well-posedness analysis is, indeed, quite standard, several refined estimates in the course of proof need to be established as they are used to prove the smoothness of the data-solution map, an important key element of the abstract result by Henry \emph{et al.} Moreover, up to our knowledge, the well-posedness of equations of the form \eqref{VBL} in Sobolev spaces of $L$-periodic distributions has not been reported as such in the literature.

Once the orbital instability criterion is at hand, one may ask about its applicability to particular examples. In the aforementioned recent paper \cite{AlPl21}, the authors applied dynamical systems techniques in order to show that, under certain structural assumptions, there exist two families of periodic waves for equations of the form \eqref{VBL}. The first family emerges from a local Hopf bifurcation when the speed $c$ crosses a critical value $c_0$. These waves have small-amplitude and finite period. The second family is generated by a global homoclinic bifurcation around a second critical value of the speed $c_1$, which is the speed of a traveling pulse or homoclinic wave. These periodic waves have amplitude of order $O(1)$ but have large period tending to $\infty$ (which can be regarded as the period of the traveling pulse). A couple of examples and numerical computations, which illustrate both families of waves, are also presented. Therefore, in order to present the applicability of the criterion, we study the orbital instability of both families of waves via a verification of the conditions for an unstable spectrum. The two families are parametrized by small parameter $\ep > 0$ (measuring the deviation of the speed of the wave from the critical speed in each case). Hence, we obtain instability under the flow of the evolution equation in periodic spaces with same period of the wave, once the parameter $\ep > 0$ is fixed (see Theorems  \ref{teoorbsmall} and \ref{teoorblarge} below). 

The paper is structured as follows. In Section \ref{secprelim} we make precise the notions of spectral and orbital instability of periodic waves and state the main results of the paper, namely, the orbital instability criterion and the well-posedness theorem. Section \ref{secwellpos} is devoted to the well-posedness theory for equations of the form \eqref{VBL} in periodic Sobolev spaces. Special attention is devoted to show that the data-solution map is smooth enough. Section \ref{secmain} contains the proof that spectral instability implies orbital instability, upon application of an abstract result on instability of equilibrium points. The final Section \ref{secappl} contains the description of the two families of periodic waves found in \cite{AlPl21} and verifies the appropriate hypotheses to apply our orbital instability criterion.

%\hole{TO DO: End of section (A).}

\subsection*{On notation}
Linear operators acting on infinite-dimensional spaces are indicated with calligraphic letters (e.g., $\cL$), except for the identity operator which is indicated by $\Id$. The domain of a linear operator, $\cL : X \to Y$, with $X$, $Y$ Banach spaces, is denoted as $\cD(\cL) \subseteq X$. For a closed linear operator with dense domain the usual definitions of resolvent and spectra apply (cf. Kato \cite{Kat80}). When computed with respect to the space $X$, the spectrum of $\cL$ is denoted as $\sigma(\cL)_{|X}$. We denote the real part of a complex number $\lambda \in \C$ by $\Re\lambda$. The classical Lebesgue and Sobolev spaces of complex-valued functions on the real line will be denoted as $L^2(\R)$ and $H^m(\R)$, with $m \in \N$, endowed with the standard inner products and norms. For any $L > 0$ and any $s \in \R$, we denote by $\Hsper = \Hsper([0,L])$ the Sobolev space of $L$-periodic distributions such that
\[
\| u \|^2_s := L \sum_{k=-\infty}^{\infty} (1 + |k|^2)^s |\widehat{u}(k)|^2 < \infty,
\]
where $\widehat{u}$ is the Fourier transform of $u$. According to custom we denote $H^0_\mathrm{\tiny{per}} = \Ldper$. If $s > k + \tfrac{1}{2}$, $k \in \N \cup \{0\}$, then there holds the continuous embedding, $\Hsper \hookrightarrow C^k_\mathrm{\tiny{per}}$, where $C^k_\mathrm{\tiny{per}}$ is the space of $L$-periodic functions with $k$ continuous derivatives. The translation operator in $\Hsper([0,L])$ will be denoted as $\zeta_\eta : \Hsper([0,L]) \to \Hsper([0,L])$, $\zeta_\eta(u) = u(\cdot + \eta)$ for any $\eta \in \R$. Translation is a smooth operator in $\Hsper([0,L])$. Moreover, if $s \geq 0$ then we have $\| \zeta_\eta(u) \|_s = \| u \|_s$ for all $u \in \Hsper$ and all $\eta \in \R$ (see Iorio and Iorio \cite{IoIo01} for details).

\section{Stability Framework and Main Theorems}
\label{secprelim}
In this section  we describe the different notions of stability under consideration and state our main results.
 
\subsection{Spectral stability}

Suppose that a sufficiently smooth profile function, $\varphi = \varphi(\cdot)$, determines an $L$-periodic traveling wave solution to \eqref{VBL} of the form \eqref{tws} for some speed value $c \in \R$. Substitution of \eqref{tws} into \eqref{VBL} yields the following ODE for the profile,
\begin{equation}
\label{profileq}
-c \varphi' + f'(\varphi)\varphi' = \varphi'' + g(\varphi).
\end{equation}
With a slight abuse of notation let us rescale the space variable as $x \mapsto x - c t$ (the co-moving Galilean frame) in order to transform \eqref{VBL} into the equation
\begin{equation}
\label{VBLg}
u_t = u_{xx} + g(u) +cu_x - f(u)_x,
\end{equation}
for which now the periodic wave is a stationary solution, $u(x,t) = \varphi(x)$, in view of \eqref{profileq}. For solutions to \eqref{VBLg} of the form $\varphi(x) + v(x,t)$, where $v$ denotes a nearby perturbation, the leading approximation is given by the linearization of this equation around $\varphi$, namely
\[
v_t = v_{xx} + (c - f'(\varphi))v_x + (g'(\varphi) - f'(\varphi)_x) v.
\]
Specializing to perturbations of the form $v(x,t) = e^{\lambda t} u(x)$, where $\lambda \in \C$ and $u$ lies in an appropriate Banach space $X$, we arrive at the eigenvalue problem
\begin{equation}
\label{primeraev}
\lambda u = u_{xx} + (c - f'(\varphi))u_x + (g'(\varphi) - f'(\varphi)_x) u,
\end{equation}
in which the complex growth rate appears as the eigenvalue. Intuitively, a necessary condition for the wave to be ``stable" is the absence of eigenvalues with $\Re \lambda > 0$, precluding exponentially growing models at the linear level. Motivated by the notion of spatially localized, finite energy perturbations in the Galilean coordinate frame in which the periodic wave is stationary, we consider $X = L^2(\R)$ and define the linearized operator around the wave as
\begin{equation}
\label{linop}
\left\{
\begin{aligned}
\cL^c \, &: \, L^2(\R) \longrightarrow L^2(\R),\\
\cL^c \, &: = \, \partial_x^2 + a_1(x) \partial_x + a_0(x) \Id, 
\end{aligned}
\right.
\end{equation}
with dense domain $\cD(\cL^c) = H^2(\R)$, and where the coefficients,
\begin{equation}
\label{defas}
\begin{aligned}
a_1(x) &:= c - f'(\varphi),\\
a_0(x) &:= g'(\varphi) - f'(\varphi)_x,
\end{aligned}
\end{equation}
are bounded and periodic, satisfying $a_j(x + L) = a_j(x)$ for all $x \in \R$, $j = 0,1$. $\cL^c$ is a densely defined, closed operator acting on $L^2(\R)$ with domain $\cD(\cL^c) = H^2(\R)$. Hence, the eigenvalue problem \eqref{primeraev} is recast as $\cL^c u = \lambda u$ for some $\lambda \in \C$ and $u \in \cD(\cL^c) = H^2(\R)$.
%
%Let us recall that, for a $\cL : X \to Y$ closed linear operator, with $X, Y$ Banach spaces and dense domain $\cD(\cL) \subset X$, the $X$-resolvent set of $\cL$, denoted as $\rho(\cL)$, is defined as the set of complex numbers $\lambda \in \C$ such that $\cL - \lambda \Id$ is invertible and $(\cL - \lambda \Id)^{-1}$ is a bounded operator. The complement of $\rho(\cL)$ is what we call the $X$-spectrum of $\cL$ and we denote it as $\sigma(\cL) = \C \backslash \rho(\cL)$. We denote the spectrum of an operator $\cL$ when computed with respect to the space $X$ as $\sigma(\cL)_{|X}$.

\begin{definition}[spectral stability]
\label{defspectstab}
We say that a bounded periodic wave $\varphi$ is \emph{spectrally stable} as a solution to the viscous balance law \eqref{VBL} if the $L^2$-spectrum of the linearized operator around the wave defined in \eqref{linop} satisfies
\[
\sigma(\cL)_{|L^2(\R)} \cap \{\lambda \in \C \, : \, \Re \lambda > 0\} = \varnothing.
\]
Otherwise we say that it is \emph{spectrally unstable}.
\end{definition}

\begin{remark}
We remind the reader that any complex number $\lambda$ belongs to the point spectrum of an operator $\cL$, denoted as $\ptsp(\cL)$, if $\cL - \lambda$ is a Fredholm operator with index equal to zero and with a non-trivial kernel. $\lambda$ belongs to the essential spectrum, $\ess(\cL)$, provided that either $\cL - \lambda$ is not Fredholm, or it is Fredholm with non-zero index. Clearly, $\ptsp(\cL), \ess(\cL) \subset \sigma(\cL)$. Moreover, since the operator is closed, $\sigma(\cL) = \ptsp(\cL) \cup \ess(\cL)$. The point spectrum consists of isolated eigenvalues with finite (algebraic) multiplicity (see \cite{Kat80,KaPro13} for further information).
\end{remark}

Since the coefficients of the operator $\cL^c$ are periodic, it is well known from Floquet theory that $\cL^c$ has no $L^2$-point spectrum and that its spectrum is purely essential (or continuous), $\sigma(\cL^c)_{|L^2(\R)} = \ess(\cL^c)_{|L^2(\R)}$ (see Lemma 3.3 in \cite{JMMP14}, or Lemma 59, p. 1487, in \cite{DunSch2}). However, it is possible to parametrize the spectrum in terms of Floquet multipliers of the form $e^{i\theta} \in \bbS^1$, $\theta \in \R$ (mod $2\pi$) via a \emph{Bloch transformation} \cite{KaPro13,Grd97}. Indeed, the purely essential spectrum $\sigma(\cL^c)_{|L^2(\R)}$ can be written as the union of partial point spectra:
\begin{equation}
\label{Floquetrep}
\sigma(\cL^c)_{|L^2(\R)} =  \!\!\bigcup_{-\pi<\theta \leq \pi}\ptsp(\cL^c_\theta)_{|\Ldper([0,L])},
\end{equation}
where the one-parameter family of Bloch operators,
\begin{equation}
\label{Blochop}
\left\{
\begin{aligned}
\cL^c_\theta &:= (\partial_x + i\theta/L)^2 + a_1(x) (\partial_x + i \theta/L) + a_0(x) \Id,\\
\cL^c_\theta &: \Ldper([0,L]) \to \Ldper([0,L]),
\end{aligned}
\right.
\end{equation}
with domain $\cD(\cL^c_\theta) = \Hdper([0,L])$, are parametrized by the Floquet exponent (or \emph{Bloch parameter}) $\theta \in (-\pi,\pi]$, and act on the periodic Sobolev space with same period $L > 0$ as the period of the wave. Since the family has compactly embedded domains in $\Ldper = \Ldper([0,L])$ then their spectrum consists entirely of isolated eigenvalues, $\sigma(\cL^c_\theta)_{|\Ldper} = \ptsp(\cL^c_\theta)_{|\Ldper} $. Moreover, they depend continuously on the Bloch parameter $\theta$, which may be regarded as a local coordinate for the spectrum $\sigma(\cL^c)_{|L^2(\R)}$ (see Proposition 3.7 in \cite{JMMP14}), meaning that $\lambda \in \sigma(\cL^c)_{|L^2(\R)}$ if and only if $\lambda \in \ptsp(\cL^c_\theta)_{|\Ldper}$ for some $\theta \in (-\pi,\pi]$. The parametrization \eqref{Floquetrep} is called the \emph{Floquet characterization of the spectrum} (for details, see \cite{AlPl21,JMMP14,KaPro13,Grd97} and the references therein). As a consequence of \eqref{Floquetrep} we conclude that the periodic wave $\varphi$ is $L^2$-spectrally unstable if and only if there exists $\theta_0 \in (-\pi, \pi]$ for which
\[
\ptsp(\cL_{\theta_0}^\ep)_{|\Ldper([0,L])} \cap \{ \lambda \in \C \, : \, \Re \lambda > 0\} \neq \varnothing.
\]

\begin{remark}
Notice that when the Bloch parameter is $\theta = 0$, the expression of the operator $\cL_0^c$ coincides with that of the linearized operator around the wave in \eqref{linop}, but now acting on a periodic space:
\begin{equation}
\label{linopBloch0}
\left\{
\begin{aligned}
\cL_0^c \, &: \, \Ldper([0,L]) \longrightarrow \Ldper([0,L]),\\
\cL_0^c \, &: = \, \partial_x^2 + a_1(x) \partial_x + a_0(x) \Id.
\end{aligned}
\right.
\end{equation}
\end{remark}

\subsection{Orbital stability}

Once the spectral (in)stability of a periodic wave is established, a natural question arises. Is the traveling wave solution nonlinearly stable, in a certain sense, with respect to the flow of the equation \eqref{VBL}? Can we deduce from the spectral (in)stability of a periodic wave a nonlinear (in)stability result? In this paper we prove that spectral instability implies orbital (nonlinear) instability in a sense that is described below.

First, note that if the profile function $\varphi = \varphi(\cdot)$ is smooth enough then it belongs to the periodic space $\Hdper([0,L])$. Hence, one can compare the motion $\varphi(x - ct)$, as a solution to \eqref{VBL}, to a general class of motions $u = u(x,t)$ evolving  from initial conditions, $u(0) =\psi$, that are close in some sense to $\varphi$. The notion of orbital stability is, thus, the property that $u(\cdot,t)$ remains close to $\varphi(\cdot +\gamma)$, $\gamma=\gamma(t)$, for all times provided that $u(0)$ starts close to $\varphi(\cdot)$. In other words, the type of stability that we expect is that the perturbation remains close to the manifold generated by translations of the traveling wave, leading to the concept of \emph{orbital stability} (also called stability \emph{in shape} \cite{AngAMS09}). We define the orbit generated by $\varphi$ as the set
\[
\cO_\varphi = \{ \varphi(\cdot + r) \, : \, r \in \R \} \subset \Hdper([0,L]).
\]
We note that $\cO_\varphi$ represents a $C^1$-curve, $\Gamma=\Gamma(r)$, in $\Hdper([0,L])$ determined by the parameter $r\in \mathbb R$, $\Gamma(r)=\zeta_r( \varphi)$. Thus, the traveling  wave profile will be orbitally stable if its orbit $\Gamma$ is stable by the flow generated by the evolution equation. Consequently, we have the following definition associated to \eqref{VBL} (cf. \cite{AngAMS09}).

\begin{definition}[orbital stability]
\label{deforbital}
Let $X$, $Y$ be Banach spaces, with the continuous embedding $Y \hookrightarrow X$. Let $\varphi \in X$ be a traveling wave solution to equation \eqref{VBL}. We say $\varphi$ is \emph{orbitally stable} in $X$ by the flow of \eqref{VBL} if for each $\vep > 0$ there exists $\delta = \delta(\vep) > 0$ such that if $\psi \in Y$ and 
\[
\inf_{r \in \R} \| \psi(\cdot) - \varphi(\cdot + r) \|_X < \delta,
\]
then the solution $u(x,t)$ of \eqref{VBL} with initial condition $u(0) = \psi$ exists globally and satisfies
\[
\sup_{t>0} \inf_{r \in \R} \| u(\cdot, t) - \varphi(\cdot + r) \|_X < \vep.
\] 
Otherwise we say that $\varphi$ is \emph{orbitally unstable} in $X$.
\end{definition}

In the present context of periodic waves for equations of the form \eqref{VBL}, we select $X = Y = \Hdper([0,L])$, where $L > 0$ is the fundamental period of the wave. 

\subsection{Main results}

Let us now state the main results of the paper. The first theorem establishes the local well-posedness of the evolution equation \eqref{VBL} in periodic Sobolev spaces.

\begin{theorem}
%[local well-posedness in periodic Sobolev spaces]
\label{teolocale}
Assume $f \in C^2(\R)$, $g \in C^1(\R)$, $L > 0$ and $s > 3/2$. If $\phi \in \Hsper([0,L])$ then there exist some ${T} = {T}(\| \phi \|_s) > 0$ and a unique solution $u \in C([0,T];\Hsper([0,L])) \cap C^1((0,T];\Hper^{s-2}([0,L]))$ to the Cauchy problem for equation \eqref{VBL} with initial datum $u(0)=\phi$. For each $T_0\in (0,T)$, the data-solution map,
\[
\phi\in \Hsper([0,L]) \mapsto u_\phi\in C([0,T_0];\Hsper([0,L])),
\]
is continuous. Moreover, if we further assume $f \in C^4(\R)$ and $g \in C^3(\R)$ then the data-solution map is of class $C^2$.
\end{theorem}

The second main result is precisely the general criterion for orbital instability based on an unstable spectrum of the linearized operator around the wave. It establishes orbital instability under the flow of the nonlinear viscous balance law in periodic Sobolev spaces with same period as the fundamental period of the wave.  

%\hole{TO DO: (B) Comments on the next theorem: the instability criterion.}

\begin{theorem}[orbital instability criterion for viscous balance laws]
\label{mainthem}
Suppose that $f \in C^4(\R)$, $g \in C^3(\R)$. Let $u(x,t) = \varphi (x-ct)$ be a periodic traveling wave solution with speed $c \in \R$ to the viscous balance law \eqref{VBL}, where the profile function $\varphi = \varphi(\cdot)$ is of class $C^2$ and has fundamental period $L > 0$. Assume that the following \emph{spectral instability property} holds: the linearized operator around the wave, $\cL_0^c : \Ldper([0,L]) \to \Ldper([0,L])$, defined in \eqref{linopBloch0}, has an unstable eigenvalue, that is, there exists $\lambda \in \C$ with $\Re \lambda > 0$ and some eigenfunction $\Psi \in \cD(\cL_0^c) = \Hdper([0,L]) \subset \Ldper([0,L])$ such that $\cL_0^c \Psi = \lambda \Psi$. Then the periodic traveling wave is orbitally unstable in $X = \Hper^2([0,L])$ under the flow of \eqref{VBL}.
\end{theorem}

%\hole{TO DO: (C) Comments on the other results: the applications.}

Additionally, in this paper we apply Theorem \ref{mainthem} to prove the nonlinear instability of periodic waves belonging to the two families described in the Introduction, whose existence and spectral instability were proved in \cite{AlPl21} (see Theorems  \ref{teoorbsmall} and \ref{teoorblarge} below).

%Complex transposition of matrices are indicated by the symbol $A^*$, whereas simple transposition is denoted by the symbol $A^\top$. For any linear operator $\mathcal{L}$, its formal adjoint is denoted by $\mathcal{L}^*$. 

%If $s > k + \tfrac{1}{2}$, $k \in \N$, then there holds the continuous embedding, $\Hsper \hookrightarrow C^k_\mathrm{\tiny{per}}$, where $C^k_\mathrm{\tiny{per}} = C^k_\mathrm{\tiny{per}}([0,L])$ is the space of periodic functions with $k$ continuous derivatives in $[0,L]$.

\section{Local well-posedness}
\label{secwellpos}

In this section we establish the local well-posedness in $\Hsper([0,L])$ for any $s > 3/2$, of the model equation \eqref{VBL}. Our analysis is standard and it is based on Banach's fixed point theorem. Albeit the arguments are classical and without major problems, several estimates are key ingredients in order to obtain the smoothness of the data-solution map associated to \eqref{VBL} (see Section \ref{secsmooth}). In the sequel (and for the rest of the paper) we use the notation,
\[
X_s=\Hsper([0,L]), \qquad \text{for any } s \in \R.
\] 

We start with the recollection of well-known facts.

\subsection{The heat semigroup in $X_s$}

The following properties of the heat semigroup acting on periodic Sobolev spaces can be found in the book by Iorio and Iorio \cite{IoIo01}.
\begin{theorem}
\label{teoheatSG}
The Cauchy problem for the heat equation,
\begin{equation}
\label{heat}
\begin{aligned}
u_t& = u_{xx} , \\
u(0) &= \phi, 
\end{aligned}
\end{equation}
is globally well-posed in $X_s$ for any $s \in \R$, $L > 0$. That is, if $\phi \in X_s$ then there exists a unique mild solution $u \in C([0,T];X_s)$ for all $T > 0$. The solution is given by
\[
u(t) = \V(t) \phi,
\]
where the family of operators, $\V(t) : X_s \to X_s$, $t \geq 0$, is the heat $C_0$-semigroup of contractions,
\[
\V(t) \phi := \big( e^{-k^2 t} \widehat{\phi} \, \big)^{\vee},
\]
with generator $\cT = \partial_x^2$ and dense domain $\cD = \Hsdper([0,L])$. The solution depends continuously on the initial data in the following sense,
\[
\sup_{t\in[0, \infty)} \big\| \V(t) \phi_1 - \V(t) \phi_2 \big\|_s \leq \| \phi_1 - \phi_2\|_s.
\]
\end{theorem}
\begin{proof}
Follows from standard theory: it is a particular case (with $q=0$) of Corollary 4.16 and Theorems 4.9, 4.14 and 4.25 in Iorio and Iorio \cite{IoIo01}, pp. 218--232.
\end{proof}

\begin{corollary}
\label{corheat}
For all $s \in \R$, $N \geq 0$ and any $\phi \in \Hsper$,
\begin{equation}
\label{limsmN}
\lim_{h \to 0} \left\| h^{-1} (\V(t+h) - \V(t))\phi - \partial_x^2 (\V(t) \phi ) \right\|_{s-N} = 0,
\end{equation}
uniformly with respect to $t \geq 0$. In particular, there exists a uniform $\overline{C} > 0$ such that
\begin{equation}
\label{unifcota}
\| h^{-1} (\V(h) - \Id) - \partial_x^2 \| \leq \overline{C},
\end{equation}
in the operator norm for all small $0 < |h| \ll 1$.
\end{corollary}
\begin{proof}
See Theorem 4.15 in \cite{IoIo01}. The second assertion follows immediately from \eqref{limsmN}.
\end{proof}

\begin{corollary}[regularity inequality]
\label{corregineq}
For all $r \in \R$ and $\delta \geq 0$ there exists a uniform constant $K_\delta > 0$ depending only on $\delta$ such that
\begin{equation}
\label{regineq}
\| \V(t) u \|_{r + \delta} \leq K_\delta \Big[ 1 + \Big( \frac{\delta}{2t}\Big)^{\delta}\Big]^{1/2} \| u \|_r,
\end{equation}
for all $u \in \Hper^r([0,L])$, $t >0$.
\end{corollary}
\begin{proof}
This is a particular case, with $q = 0$ and $\mu =1$, of Theorem 4.17 in \cite{IoIo01}.
\end{proof}

\subsection{The Cauchy problem for viscous balance laws}
Let us consider the Cauchy problem for the viscous balance law \eqref{VBL} in $X_s$, $s > 3/2$. It reads:
\begin{equation}
\label{CpVBL}
\begin{aligned}
u_t - u_{xx} &= g(u) - f'(u)u_x,\\
u(0) &= \phi,
\end{aligned}
\end{equation}
for some initial condition $u(0) = \phi \in X_s$. Assuming $f \in C^2$, $g \in C^1$, we define $F \in C^1(\R^2)$ as 
\[
F(u,p) := g(u) - f'(u)p, \qquad (u,p) \in \R^2.
\]
Hence, the Cauchy problem \eqref{CpVBL} can be recast as
\begin{equation}
\label{CpVBL2}
\begin{aligned}
u_t - u_{xx} &= F(u,u_x),\\
u(0) &= \phi,
\end{aligned}
\end{equation}
with $\phi \in X_s$. Upon application of the variation of constants formula we arrive at the integral equation,
\begin{equation}
\label{inteq}
\cA u := u(t) = \V(t) \phi + \int_0^t \V(t-\tau) F(u,u_x) \, d \tau.
\end{equation}

In order to prove existence and uniqueness of solutions to the Cauchy problem \eqref{CpVBL2} we follow the standard blueprint (see, e.g. Taylor \cite{TayPDE3-2e}, chapter 15): (i) the linear part of the equation generates a $C_0$ semigroup in a certain Banach space $X$ (this step has been already verified by Theorem \ref{teoheatSG}); (ii), the nonlinear term $F$ is locally Lipschitz from $X$ to another Banach space $Y$; and, (iii), the operator $\cA$ is a contraction in a closed ball in $C([0,T];X)$ for $T$ sufficiently small yielding, upon application of Banach's fixed point theorem, a solution to the integral equation \eqref{inteq}.

The following lemmata are devoted to verify these steps in the context of periodic Sobolev spaces and in the spirit of the analysis of Iorio and Iorio \cite{IoIo01} for nonlinear equations.

\begin{lemma}
\label{lemFlip}
For any $s > 3/2$ and assuming $f \in C^2$, $g \in C^1$, then $F = F(u,u_x)$ is locally Lipschitz from $\Hsper$ to $\Hper^{s-1}$. More precisely, for any 
\[
u,v \in \overline{B_M} = \{ w \in \Hsper \, : \, \| w \|_s \leq M \} \subset \Hsper,
\]
with $M > 0$ fixed but arbitrary, there holds the estimate
\begin{equation}
\label{E1}
\| F(u,u_x) - F(v,v_x) \|_{s-1} \leq L_s(\| u \|_s, \| v \|_s) \| u - v \|_s,
\end{equation}
where $L_s : [0,\infty) \times [0,\infty) \to (0,\infty)$, $L_s = L_s(\varrho_1,\varrho_2)$, is a continuous, positive function and non-decreasing with respect to each argument. In particular, there holds the estimate
\begin{equation}
\label{E2}
\| F(u,u_x) \|_{s-1} \leq L_s(\| u \|_s, 0) \| u  \|_s,
\end{equation}
for all $u \in \overline{B_M}$.
\end{lemma}
\begin{proof}
Let $u,v \in \overline{B_M}$. Since for each $s > 1/2$, $\Hsper$ is a Banach algebra (see Theorem 3.200 in \cite{IoIo01}), there exists a constant $C_s \geq 0$ depending only on $s$ such that
\[
\begin{aligned}
\| F(u,u_x) - F(v,v_x) \|_{s-1} %&= \| g(u) - g(v) - f'(u)(u_x - v_x) + (f'(v) - f'(u)) v_x \|_{s-1} \\
&\leq \| g(u) - g(v) \|_{s-1} + C_s \| f'(u) \|_{s-1} \| u_x - v_x \|_{s-1} + \\ &\quad + C_s \|f'(v) - f'(u) \|_{s-1} \| v_x \|_{s-1}\\
&\leq \| g(u) - g(v) \|_{s-1} + C_s \| f'(u) \|_{s-1} \| u - v \|_{s} + \\ &\quad + C_s \|f'(v) - f'(u) \|_{s-1} \| v \|_{s}.
\end{aligned}
\]

In view that $s > 3/2$ we have $\Hsper \subset \Huper$ and by Sobolev's inequality there holds $|u| \leq \| u \|_{L^\infty} \leq 2 \| u \|_0^{1/2}  \| u_x \|_0^{1/2} \leq 2 \|u\|_1 \leq 2 \|u\|_s \leq 2M$ a.e. in $x \in [0,L]$ for all $u \in \overline{B_M}$. Since $f(u), g(u)$ and $f'(u)$ are continuous in the compact set $[-2M,2M]$, then they are locally Lipschitz and there exist uniform constants $L_f, L_g > 0$, depending only on $s$ and $M$, such that 
\begin{equation}
\label{Lipest}
\begin{aligned}
\| f'(u) - f'(v) \|_{s-1} &\leq L_f \|u - v \|_{s-1},\\
\| g(u) - g(v) \|_{s-1} &\leq L_g \|u - v \|_{s-1},\\
\| g'(u) - g'(v) \|_{s-1} &\leq L_g \|u - v \|_{s-1},\\
\| f'(u) \|_{s-1} &\leq L_f \| u \|_{s-1} + |f'(0)|,\\
\| g'(u) \|_{s-1} &\leq L_g \| u \|_{s-1} + |g'(0)|,
\end{aligned}
\end{equation}
for all $u,v \in \overline{B_M}$. Therefore,
%
%
%
%$\| f'(u) - f'(v) \|_{s-1} \leq L_f \|u - v \|_s$, $\| g(u) - g(v) \|_{s-1} \leq L_f \|u - v \|_s$ and 
%$$
%\| f'(u) \|_{s-1} \leq L_f \| u \|_s + |f'(0)|,\;\;\;\;\text{for all}\;\; u,v \in \overline{B_M}.
%$$
% Thus,
\[
\begin{aligned}
\| F(u,u_x) - F(v,v_x) \|_{s-1} &\leq \big[ L_g + C_s \| f'(u) \|_{s-1}+ C_s \| v \|_{s-1} \big]  \| u - v \|_{s-1}\\
&\leq \big[ L_g + C_s (L_f \| u \|_s + |f'(0)|+ \| v \|_s) \big]  \| u - v \|_{s}\\
&\leq L_s(\|u\|_s, \|v \|_s) \| u - v \|_{s},
\end{aligned}
\]
since from definition $\| u \|_{s-1} \leq \| u \|_s$ for all $u$, and where we have defined\footnote{for later use (see the proof of Lemma \ref{leminvert} below) we have incorporated $|g'(0)|$ into the definition of this upper bound $L_s(\cdot,\cdot)$.}
\begin{equation}
\label{defLs}
L_s(\varrho_1,\varrho_2) := L_g + C_s \big( (L_f + L_g) \varrho_1 + \varrho_2 + |f'(0)| + |g'(0)| \big) > 0, 
%\quad (\varrho_1,\varrho_2) \in [0,\infty) \times [0,\infty).
\end{equation}
for all $(\varrho_1,\varrho_2) \in [0,\infty) \times [0,\infty)$. Clearly, $L_s(\cdot,\cdot)$ is continuous and non-decreasing with respect to each argument. This yields \eqref{E1} and the lemma is proved.
%with $L_s(\varrho_1,\varrho_2) := L_g + C_s (L_f \varrho_1 + \varrho_2 + |f'(0)|)$, for $\varrho_j \geq 0$.
\end{proof}

Let $\phi \in X_s$. For any $\alpha > 0$, fixed but arbitrary, and for $T > 0$ to be chosen later, let us define
\[
Z_{\alpha,T} := \Big\{ u \in C([0,T];X_s) \, : \, \sup_{t \in [0,T]} \| u(t) - \phi \|_s \leq \alpha \Big\}.
\]
Clearly, $Z_{\alpha,T} \subset C([0,T];X_s)$ and it is closed under the norm
\[
\| u \|_{C([0,T]; X_s)} := \sup_{t \in [0,T]} \| u(t) \|_s.
\]

Next result establishes the conditions under which the operator $\cA$ defined in \eqref{inteq} is a contraction mapping on $Z_{\alpha,T}$, yielding the existence and uniqueness of a mild solution to \eqref{CpVBL2}.

\begin{lemma}
\label{lemmild}
Let $s > 3/2$ and $\phi \in X_s$. Then there exist $T > 0$ and a unique mild solution $u \in C([0,T];X_s)$ to the Cauchy problem \eqref{CpVBL2} (that is, to the integral equation \eqref{inteq}). Moreover, the data-solution map $\phi \mapsto u$ is continuous.
\end{lemma}

\begin{proof}
First, we verify that if $u \in Z_{\alpha,T}$ then $\cA u \in C([0,T];X_s)$. Indeed, for all $0 < t_1 < t_2 < T$ we have
\begin{equation}
\label{difA}
\begin{aligned}
\| \cA u(t_1) - \cA u(t_2) \|_s &\leq \| (\V(t_1) - \V(t_2) ) u \|_s + \int_0^{t_1} \| (\V(t_1-\tau) - \V(t_2-\tau)) F(u,u_x) \|_s \, d\tau\\
& \quad + \int_{t_1}^{t_2} \| (\V(t_1-\tau) - \V(t_2-\tau)) F(u,u_x) \|_s \, d\tau.
\end{aligned}
\end{equation}

Since $\V(t)$ is a $C_0$-semigroup, $\| (\V(t_1) - \V(t_2) ) u \|_s\to 0$ as  $t_2\to t_1$. In order to control the second term in \eqref{difA}, we apply  inequality \eqref{regineq} with $\delta = 1$, $r = s-1$, $C = K_1 > 0$, and estimate \eqref{E2} (inasmuch as $Z_{\alpha,T} \subset \overline{B_M}$ with $M = \alpha + \| \phi \|_s$); this yields,
\[
\begin{aligned}
\| (\V(t_1-\tau) - \V(t_2-\tau)) F(u,u_x) \|_s \leq 2C \Big[ 1 + \frac{1}{2(t_1-\tau)} \Big]^{1/2} \!\! \sup_{t \in [0,T]} L_s(\| u(t) \|_s,0) \| u(t) \|_s,
\end{aligned}
\]
for all $0 < \tau \leq t_1$. The function on the right side of last inequality is integrable in $\tau \in (0,t_1)$. Therefore, by the Dominated Convergence Theorem,
\[
\lim_{t_2 \to t_1} \int_0^{t_1} \| (\V(t_1-\tau) - \V(t_2-\tau)) F(u,u_x) \|_s \, d\tau = 0.
\]
Analogously, for the second integral in \eqref{difA} we have the estimate
\[
\|\V(t_2-\tau) F(u,u_x) \|_s \leq C \Big[ 1 +  \frac{1}{2(t_2-\tau)} \Big]^{1/2} L_s(\| u(\tau) \|_s,0) \| u(\tau) \|_s,
\]
for all $\tau \in (t_1,t_2)$.  Clearly, since $u \in Z_{\alpha,T}$ then we have $\| u(\tau)\|_s \leq \alpha + \| \phi \|_s$ and therefore
\[
\begin{aligned}
\int_{t_1}^{t_2} \| \V(t_2-\tau) F(u,u_x) \|_s \, d\tau \leq C L_s(\alpha + \| \phi  \|_s,0) ( \alpha + \| \phi \|_s) \big( t_2 - t_1 - \sqrt{2(t_2-t_1)}\big) \to 0,
\end{aligned}
\]
as $t_2 \to t_1$. This shows that $\cA u(t) \in X_s$ for all $t \in [0,T]$ and that $\cA u \in C([0,T];X_s)$.

Next, we choose $T> 0$ small enough such that $\cA(Z_{\alpha,T}) \subset Z_{\alpha,T}$ and that $\cA$ is a contractive mapping. First, note that since $\V(t)$ is a $C_0$-semigroup we can choose $T_1 > 0$ such that $\| \V(t) \phi - \phi \|_s < \alpha/2$ for all $t \in [0,T_1]$. Now, if $u \in Z_{\alpha,T}$ then we have the  estimate (see \eqref{E2}),
\[
\begin{aligned}
\Big\| \int_0^t \V(t - \tau) F(u, u_x) \, d\tau \Big\|_s &\leq \int_0^t \| \V(t-\tau) F(u,u_x) \|_s \, d\tau \\
&\leq C L_s(\alpha + \| \phi  \|_s,0) ( \alpha + \| \phi \|_s) (T + \sqrt{2 T}) \\
&< \tfrac{1}{2} \alpha,
\end{aligned}
\]
provided that we choose $T < T_1$ small enough. This shows that $\cA(Z_{\alpha,T}) \subset Z_{\alpha,T}$. Finally, in order to show that $\cA$ is a contraction for some (possibly smaller) $T > 0$, let $u,v \in Z_{\alpha,T}$. Similar arguments yield the estimate
\[
\begin{aligned}
\| \cA u(t) - \cA v(t) \|_s &\leq \int_0^t \| \V(t-\tau)(F(u,u_x) - F(v,v_x)) \|_s \, d\tau \\
&\leq C \int_0^t \Big[ 1 + \frac{1}{2(t - \tau)} \Big]^{1/2} \| F(u,u_x) - F(v, v_x) \|_s \, d\tau \\
&\leq C L_s( \alpha + \| \phi \|_s, \alpha + \| \phi \|_s )  (T + \sqrt{2 T}) \sup_{t \in [0,T]} \| u(t) - v(t) \|_s\\
&< \tfrac{1}{2} \| u - v \|_{C([0,T];X_s)},
\end{aligned}
\]
where for
\begin{equation}
\label{C}
C_\phi :=  C L_s( \alpha + \| \phi \|_s, \alpha + \| \phi \|_s ) > 0,
\end{equation}
we choose $T$ sufficiently small such that 
\begin{equation}
\label{laii}
C_\phi (T + \sqrt{2 T}) < \frac{1}{2}.
\end{equation}
Notice that $T$ depends on $\| \phi \|_s$. Hence, we conclude that there exists a small ${T} ={T}(\| \phi \|_s) > 0$ such that $\cA(Z_{\alpha,T}) \subset Z_{\alpha,T}$ and $\cA$ is a contraction on $Z_{\alpha,T}$. By Banach's fixed point theorem, there exists a unique fixed point $u \in Z_{\alpha,T}$ of $\cA$ that solves \eqref{inteq}. 

Finally, to show the continuity of the data-solution map let $u$ and $v$ in $C([0,T];X_s)$ be the solutions to the Cauchy problem with initial data $u(0) = \phi$ and $v(0) = \psi$, respectively. Then, using the regularity estimate \eqref{regineq} it is easy to show that
\[
\begin{aligned}
\| u(t) - v(t) \|_s &\leq \| \cV(t) \phi - \cV(t) \psi \|_s + \int_0^t \| \cV(t-\tau) (F(u,u_x) - F(v,v_x)) \|_s \, d\tau\\
&\leq \| \phi -  \psi \|_s + (T + \sqrt{2T}) L_s(M_s,M_s) \int_0^t \| u(\tau) - v(\tau) \|_s \, ds,
\end{aligned}
\]
with $M_s := \max \,  \big\{ \sup_{t \in [0,T]} \| u(t) \|_s, \sup_{t \in [0,T]} \| u(t) \|_s \big\}$. Gronwall's inequality then yields 
\[
\| u(t) - v(t) \|_s \leq C_{s,T} \| \phi -  \psi \|_s,
\]
for all $t \in [0,T]$ with a constant $C_{s,T} > 0$ depending only on $s$ and $T$. The lemma is proved.
\end{proof}

It remains to verify that the unique solution from Lemma \ref{lemmild} is, in fact, a strong solution to the Cauchy problem \eqref{CpVBL}.
%
%In the following we show one part of Theorem \ref{teolocale}, the smoothness property of the data-solution map will be show in  Lemma \ref{leminvert} below.
%

\begin{lemma}
\label{lemstrong}
Under the assumptions of Lemma \ref{lemmild}, the unique mild solution $u \in C([0,T];X_s)$ to \eqref{inteq} satisfies $u \in C^1([0,T];\Hper^{s-2}([0,L]))$ and, therefore, it is a strong solution to the Cauchy problem \eqref{CpVBL}.
\end{lemma}

\begin{proof}
%[Theorem \ref{teolocale}]
%It remains to verify that the unique solution from Lemma \ref{lemmild} is a strong solution to the Cauchy problem \eqref{CpVBL}, that is, $u \in C^1([0,T];\Hper^{s-2}([0,L]))$. 
It suffices to show that
\[
\lim_{h \to 0} \| h^{-1} (u(t+h) - u(t)) - \partial_x^2 u - F(u, u_x) \|_{s-2} = 0.
\]
To that end, write
\begin{equation}
\label{ABC}
\begin{aligned}
h^{-1} (u(t+h) &- u(t)) - \partial_x^2 u - F(u, u_x) = h^{-1} (\V(t+h) - \V(t)) \phi - \partial_x^2 (\V(t) \phi) \\
&+ h^{-1} \int_0^t (\V(t+h-\tau) - \V(t-\tau)) F(u, u_x)(\tau) \, d \tau - \partial^2_x \int_0^t \V(t-\tau) F(u, u_x)(\tau) \, d\tau \\
&+ h^{-1} \int_t^{t+h} \V(t+h-\tau) F(u,u_x)(\tau) \, d \tau - F(u, u_x)(t).
\end{aligned}
\end{equation}
First, note that Corollary \ref{corheat} immediately implies that
\[
\lim_{h \to 0} \| h^{-1} (u(t+h) - u(t)) - \partial_x^2 (\V(t) \phi)  \|_{s-2} = 0.
\]
The $\| \cdot \|_{s-2}$-norm of the last term in \eqref{ABC} is clearly bounded above by
\[
h^{-1}\int_t^{t+h} R(\tau) \, d \tau,
\]
where $R(\tau) := \| \V(t+h-\tau) F(u,u_x)(\tau) - F(u, u_x)(t) \|_{s-2}$. $R$ is a continuous function of $\tau \in (t, t + h)$ and, hence, there exists some $\vartheta \in (t, t + h)$ for which
\[
R(\vartheta) = h^{-1}\int_t^{t+h} R(\tau) \, d \tau.
\]
Since $\vartheta \to t$ as $h \to 0$, by continuity of the semigroup we have,
\[
\lim_{h \to 0} R(\vartheta) = \lim_{h \to 0} \| \V(t+h-\vartheta) F(u,u_x)(\vartheta) - F(u, u_x)(t) \|_{s-2} = 0.
\]
This yields
\[
0 \leq \lim_{h \to 0} \Big\| h^{-1} \int_t^{t+h} \V(t+h-\tau) F(u,u_x)(\tau) \, d \tau - F(u, u_x)(t) \Big\|_{s-2} \leq \lim_{h \to 0} R(\vartheta) = 0.
\]

Finally, apply \eqref{regineq}, \eqref{unifcota} and \eqref{E2} to observe that, for all $ 0 < \tau < t$ and all $|h|$ small, there holds the estimate
\[
\begin{aligned}
\big\| h^{-1} (\V(t+h &-\tau) - \V(t-\tau)) F(u, u_x)(\tau) - \partial^2_x \big( \V(t-\tau) F(u, u_x)(\tau)\big) \big\|_{s-2} =\\
&\leq \| \V(t-\tau)( h^{-1} (\V(h) - \Id) - \partial^2_x )F(u, u_x)(\tau) \|_{s-2}\\
&\leq C \Big[ 1 + \frac{1}{2(t-\tau)}\Big]^{1/2} \| h^{-1} (\V(h) - \Id) - \partial^2_x \| \| F(u, u_x)(\tau) \|_{s-3}\\
&\leq C \overline{C} \Big[ 1 + \frac{1}{2(t-\tau)}\Big]^{1/2}  \| F(u, u_x)(\tau) \|_{s-1}\\
&\leq C \overline{C} L_s(\sup_{\tau \in(0,t)} \| u(\tau) \|_s,0)) \sup_{\tau \in (0,t)} \| u(\tau) \|_s \Big[ 1 + \frac{1}{2(t-\tau)}\Big]^{1/2}.
\end{aligned}
\]
Once again, the right hand side of last inequality is integrable in $\tau \in (0,t)$.  Corollary \ref{corheat} then yields
\[
\big\| (h^{-1}(\V(h) - \Id) - \partial^2_x )F(u, u_x)(\tau) \big\|_{s-2} \to 0,
\]
as $h \to 0$, uniformly in $\tau \in (0,t)$. Thus, by the Dominated Convergence Theorem, we conclude that
\[
\lim_{h \to 0} \Big\| h^{-1} \int_0^t (\V(t+h-\tau) - \V(t-\tau)) F(u, u_x)(\tau) \, d \tau - \partial^2_x \int_0^t \V(t-\tau) F(u, u_x)(\tau) \, d\tau \Big\|_{s-2} = 0.
\]
This shows that $u \in C^1([0,T];\Hper^{s-2}([0,L]))$ and the lemma is proved.
%
%
%the $\| \cdot \|_{s-2}$-norm of the second term in \eqref{ABC} is bounded above by
%\[
%\int_0^t \| h^{-1} (\V(t+h-\tau) - \V(t-\tau)) F(u, u_x)(\tau) - \partial^2_x \V(t-\tau) F(u, u_x)(\tau) \|_{s-2} \, d\tau
%\]
%
%applying basic properties of semigroups and their generators, using Corollary \ref{corheat}, and applying again estimate \eqref{regineq}, it is straightforward to show that the $\| \cdot \|_{s-2}$ norm of last expression tends to zero as $h \to 0$. We omit the details. The theorem is now proved.
\end{proof}

\subsection{Smoothness of the data-solution map}

\label{secsmooth}

Let $B$ be the ball $B=B_\vep(\phi) = \{ u \in X_s \, : \, \| u - \phi \|_s < \vep \}$ with $\vep > 0$.  Define the map
\begin{equation}
\label{defGamma}
\begin{aligned}
\Gamma &: B \times C([0,T];X_s) \to C([0,T];X_s),\\
\Gamma (\psi, w)(t) &:= w(t) - \V(t) \psi - \int_0^t \V(t-\tau) F(w,w_x) \, d\tau.
\end{aligned}
\end{equation}

For any given $\phi \in X_s$, $s>3/2$, let us denote by $u_\phi \in C([0,T]; X_s)$, 
\begin{equation}
\label{uphi}
u_\phi(t) = \V(t) \phi + \int_0^t \V(t-\tau) F(u_\phi,\partial_x u_\phi ) \, d\tau,
\end{equation}
the unique solution to the Cauchy problem \eqref{CpVBL} with $u_\phi(0)=\phi$. Then, clearly, 
\[
\Gamma(\phi, u_\phi)(t) = 0,
\]
for all $t \in [0,T]$. 

At this point we need to impose further regularity on the functions $f$ and $g$ to guarantee twice Fr\'echet differentiability of the mapping $\Gamma$ in a neighborhood of $(\phi, u_\phi)$.

\begin{lemma}
\label{lemFrechdiff}
Let $f \in C^4(\R)$, $g \in C^3(\R)$ and $s > 3/2$. Then the map $\Gamma : X_s \times C([0,T];X_s) \to C([0,T];X_s)$ defined in \eqref{defGamma} is twice Fr\'echet differentiable in an open neighborhood $B_\vep(\phi) \times B_\delta(u_\phi)$ of $(\phi, u_\phi)$.
\end{lemma}
\begin{proof}
Follows directly from the regularity of $F(u,u_x) = g(u) - f'(u) u_x$, the definition of the mapping $\Gamma$ and standard properties of the contractive semigroup $\V(t)$. (Recall that the existence of continuous G\^ateaux derivatives in open neighborhoods yields Fr\'echet differentiability; see \cite{ZeidI86}, \S 4.2, Proposition 4.8.) We omit the details.
\end{proof}
\begin{lemma}
\label{leminvert}
Suppose that $f \in C^4(\R)$, $g \in C^3(\R)$. Let $\phi\in X_s$, $s>3/2$, and consider $u_\phi \in C([0,T]; X_s)$, $T>0$, the unique strong solution to  \eqref{CpVBL} given by Lemma \ref{lemstrong}. Then, the operator
\begin{equation}
\label{derivwGamma}
\begin{aligned}
\partial_w \Gamma (\phi, u_\phi) &: C([0,T];X_s) \to C([0,T];X_s),\\
\partial_w \Gamma (\phi, u_\phi) w(t) &= w(t) - \int_0^t \V(t-\tau) \Big[ \big(g'(u_\phi) - f''(u_\phi) \partial_x u_\phi\big) w - f'(u_\phi) w_x \Big] \, d\tau,
\end{aligned}
\end{equation}
is one to one and onto. Moreover, the data-solution map associated to \eqref{CpVBL},
\begin{equation}
\label{datasolmap}
\begin{aligned}
\Upsilon &: X_s \to C([0,T];X_s),\\
\phi &\mapsto \Upsilon(\phi) = u_\phi,
\end{aligned}
\end{equation}
is of class $C^2$.
\end{lemma}
\begin{proof}
First, let us verify the formula for the operator $\partial_w \Gamma (\phi, u_\phi)$ by computing $\lim_{h \to 0} h^{-1} \Gamma(\phi, u_\phi + hw)$ for any $w \in C([0,T],X_s)$. From the definition of $F = F(u,u_x)$ we have by Taylor expansion,
\[
F(u_\phi + hw, \partial_x u_\phi + hw_x) = F(u_\phi, \partial_x u_\phi) + h[ \big( g'(u_\phi) - f''(u_\phi) \partial_x u_\phi\big) w - f'(u_\phi) w_x \big)] + O(h^2).
\]
Hence,
%\[
%\begin{aligned}
%h^{-1} \Gamma(\phi, u_\phi + hw)(t) &= h^{-1} \Big[ u_\phi(t) - \V(t) \phi - \int_0^t \V(t-\tau) F(u_\phi,\partial_x u_\phi ) \, d\tau \Big] + \\
%&\quad + w(t) - \int_0^t \V(t-\tau) \Big[ \big(g'(u_\phi) - f''(u_\phi) \partial_x u_\phi\big) w - f'(u_\phi) w_x \Big] \, d\tau + O(h)\\
%&=  w(t) - \int_0^t \V(t-\tau) \Big[ \big(g'(u_\phi) - f''(u_\phi) \partial_x u_\phi\big) w - f'(u_\phi) w_x \Big] \, d\tau + O(h)
%\end{aligned}
%\]
\[
h^{-1} \Gamma(\phi, u_\phi + hw)(t) = w(t) - \int_0^t \V(t-\tau) \Big[ \big(g'(u_\phi) - f''(u_\phi) \partial_x u_\phi\big) w - f'(u_\phi) w_x \Big] \, d\tau + O(h),
\]
in view of \eqref{uphi}. This yields \eqref{derivwGamma} when $h \to 0$.

Now, let us apply the regularity inequality \eqref{regineq} to estimate, for any $t \in [0,T]$ and all $w \in C([0,T],X_s)$,
\[
\begin{aligned}
\| \partial_w \Gamma (\phi, u_\phi) w(t) - w(t) \|_s &\leq \int_0^t \Big\| \V(t-\tau) \Big[ \big(g'(u_\phi) - f''(u_\phi) \partial_x u_\phi\big) w - f'(u_\phi) w_x \Big] \Big\|_s \, d \tau \\
&\leq C \int_0^t \Big[ 1 +  \frac{1}{2(t - \tau)} \Big]^{1/2} \| g'(u_\phi) w - (f'(u_\phi)w)_x \|_{s-1} \, d \tau.
%& \leq C_\phi (T + \sqrt{2T}) \sup_{t \in [0,T]} \| w(t) \|_s\\
%&< \tfrac{1}{2} \| w \|_{C([0,T];X_s)}
\end{aligned}
\]

In view of estimates \eqref{Lipest}, the fact that $\Hsper$ is a Banach algebra for any $s > 1/2$ and that, by the fixed point theory of Lemma \ref{lemmild} there holds $\cA(Z_{\alpha,T}) \subset Z_{\alpha,T}$, or in other words, $\| u_\phi \| \leq \alpha + \| \phi \|_s$, then we obtain the following estimate for all $0 < \tau < t$,
\[
\begin{aligned}
\| g'(u_\phi) w - (f'(u_\phi)w)_x \|_{s-1} &\leq C_s \| g'(u_\phi) \|_{s-1} \| w \|_{s-1} + \| (f'(u_\phi)w)_x \|_{s-1}\\
%&\leq C_s \| g'(u_\phi) \|_{s-1} \| w \|_{s} + \| f'(u_\phi)w \|_{s}\\
&\leq C_s [ L_g \| u_\phi \|_{s-1} + |g'(0)| ] \| w \|_s + C_s [L_f \| u_\phi \|_s + |f'(0)|] \| w \|_s \\
&\leq C_s \big[ (L_f + L_g) (\alpha + \| \phi \|_s) + |f'(0)| + |g'(0)| \big] \sup_{\tau \in (0,T)} \| w(\tau) \|_s\\
&= L_s(\alpha + \| \phi \|_s, 0)  \sup_{\tau \in (0,T)} \| w(\tau) \|_s\\
&\leq L_s(\alpha + \| \phi \|_s, \alpha + \| \phi \|_s) \sup_{\tau \in (0,T)} \| w(\tau) \|_s.
\end{aligned}
\]
This yields,
\[
\begin{aligned}
\| \partial_w \Gamma (\phi, u_\phi) w(t) - w(t) \|_s &\leq C L_s(\alpha + \| \phi \|_s, \alpha + \| \phi \|_s) (T + \sqrt{T}) \sup_{\tau \in (0,T)} \| w(\tau) \|_s\\
&= C_\phi (T + \sqrt{T}) \| w \|_{C([0,T];X_s)},
\end{aligned}
\]
for the same constant $C_\phi > 0$ given in \eqref{C}. Since $T$ satisfies \eqref{laii} we conclude that 
\[
\| (\partial_w \Gamma (\phi, u_\phi)  - \Id)w \|_{C([0,T];X_s)} < \tfrac{1}{2} \| w \|_{C([0,T];X_s)},
\]
for all $w \in C([0,T],X_s)$. In other words, in the operator norm there holds 
\[
\| \partial_w \Gamma (\phi, u_\phi)  - \Id \| < 1.
\] 
This proves that $\partial_w \Gamma (\phi, u_\phi)$ is invertible on $C([0,T];X_s)$.

In view of Lemma \ref{lemFrechdiff}, we now apply the Implicit Function Theorem in Banach spaces (cf. \cite{ZeidI86}, \S 4.7) to conclude the existence of a neighborhood $\widetilde{B} \subset B$ of $\phi$ and a $C^2$-mapping
\begin{equation}
\label{datasolmap}
\Upsilon : \widetilde{B} \to C([0,T];X_s),
\end{equation}
such that $\Gamma (w, \Upsilon(w)) = 0$ for all $w \in \widetilde{B}$. By \eqref{defGamma}, the mapping $\Upsilon$ is clearly the data-solution map inasmuch as $\phi \mapsto \Upsilon(\phi) = u_\phi$. The conclusion follows.
%This finishes the proof of Theorem \ref{teolocale}.
\end{proof}

\subsection{Proof of Theorem \ref{teolocale}} Assuming $f \in C^2(\R)$, $g \in C^1(\R)$, the first assertion follows immediately upon application of Lemmata \ref{lemmild} and \ref{lemstrong}. If we suppose more regularity, $f \in C^4(\R)$ and $g \in C^3(\R)$, then the hypotheses of Lemmata \ref{lemFrechdiff} and \ref{leminvert} are satisfied and the data-solution map, $\phi \mapsto \Upsilon(\phi)$, is of class $C^2$. The theorem is proved.
\qed

\section{Orbital instability criterion}
\label{secmain}

This section is devoted to prove Theorem \ref{mainthem}.

\subsection{An abstract result}

The following theorem provides the link to obtain nonlinear (orbital) instability from spectral instability.

\begin{theorem}[Henry \emph{et al.} \cite{HPW82}]
\label{teohenry}
Let $Y$ be a Banach space and $\Omega \subset Y$ an open subset such that $0 \in \Omega$. Assume that there exists a map $\cM : \Omega \to Y$ such that $\cM(0) = 0$ and, for some $p > 1$ and some continuous linear operator $\cL$ with spectral radius $r(\cL) > 1$, there holds
\[
\| \cM(y) - \cL y \|_Y = O(\| y \|_Y^p) 
\]
as $y \to 0$. Then $0$ is unstable as a fixed point of $\cM$. More precisely, there exists $\vep_0 > 0$ such that for all $B_\eta(0) \subset Y$ and arbitrarily large $N_0 \in \N$ there is $n \geq N_0$ and $y \in B_\eta(0)$ such that $\| \cM^n(y) \|_Y \geq \vep_0$.
\end{theorem}
\begin{proof}
See Theorem 2 in \cite{HPW82} (see also Theorem 5.1.5 in \cite{He81}).
\end{proof}

\begin{remark}
\label{remHenry}
The statement in Theorem \ref{teohenry} establishes the instability of $0$ as a fixed point of $\cM$; in other words, it shows the existence of points moving away from $0$ under successive applications of $\cM$. In the Remark after Theorem 2 in  \cite{HPW82}, the following extension of Theorem 2 is obtained: if $\Gamma_0$ is a $C^1$-curve of fixed points of  $\cM$ with $0\in \Gamma_0$ then $\Gamma_0$ is unstable, in other words, the points $\{\cM^n(y), n \geq 0\}$ not only move away from $0$, but also from $\Gamma_0$.
\end{remark}

Theorem \ref{teohenry} can be recast in a more suitable form for applications to nonlinear wave instability (see also \cite{AngNat14,AngNat16}).

\begin{corollary}
\label{corhenry}
Let $\cS : \Omega \subset Y \to Y$ be a $C^2$ map defined on an open neighborhood of a fixed point $\varphi$ of $\cS$. If there is an element $\mu \in \sigma(\cS'(\varphi))$ with $|\mu| > 1$ then $\varphi$ is unstable as a fixed point of $\cS$. Moreover, if $ \Gamma$ is a $C^1$-curve of fixed points of  $\cS$ with $\varphi \in \Gamma$ then $\Gamma$ is unstable.
\end{corollary}
\begin{proof}
Define the open set $\widetilde{\Omega} = \{ y - \varphi \, : \, y \in B \} \subset Y$, where $B = B_\delta(\varphi)$ is an open ball with radius $\delta > 0$, and consider the mapping $\cM : \widetilde{\Omega} \to Y$, $\cM(x) := \cS(x+\varphi) - \varphi$. Then, clearly, $\cM(0) = 0$ and $\cM$ is of class $C^2$ in $\widetilde{\Omega}$. Define $\cZ := \cS'(\varphi)$. Then, by hypothesis, there exists an eigenvalue $\mu \in \sigma(\cZ)$ with $1 < |\mu| \leq r(\cZ)$. By Taylor's formula,
\[
\cM(x) = \cM(0) + \cM'(0) x + O(\|x\|_Y^2) = \cZ x + O(\|x\|_Y^2),
\]
provided that $\| x \|_Y \ll 1$. Apply Theorem \ref{teohenry} to deduce the existence of $\vep_0 > 0$ such that, for any ball $B_\eta(\varphi)$, with radius $\eta > 0$ and arbitrarily large $N_0 \in \N$, there exists $n \geq N_0$ and $y \in B_\eta(\varphi)$ such that $\| \cS^n(y) - \varphi \|_Y \geq \vep_0$. This completes the proof.
\end{proof}

 %VERSION 2
%\begin{proof}
%Follows directly from Theorem \ref{teohenry} (see, e.g., the proof of Corollary 3.1 in \cite{AngNat16}).
%\end{proof}

\subsection{The mapping $\cS$}

Before proving our main result, Theorem \ref{mainthem}, we need to specify the particular mapping $\cS$ (in the context of Corollary \ref{corhenry}) suitable for our needs. We start by making a couple of observations.

First notice that, if we denote the unique solution to the Cauchy problem \eqref{CpVBL} with initial datum $\varphi \in X_2 = \Hdper$ as $u_\varphi = \Upsilon(\varphi) \in C([0,T];X_2)$ (where $\varphi = \varphi(\cdot)$ is the $L$-periodic $C^2$ profile function), then for each $x \in [0,L]$ a.e. there holds $u_\varphi(t)(x) = \varphi(x-ct)$, or, in other words, 
\begin{equation}
\label{solprofil}
u_\varphi(t) = \varphi(\cdot - ct) = \zeta_{-ct}(\varphi) \in X_2, 
\end{equation}
where $\zeta_\eta$ is the translation operator in $X_2$ for any $\eta \in \R$. This follows by direct differentiation and by the profile equation \eqref{profileq}.

%
%(Yo pienso que esta parte sobra o deveria ser escrito algo bien corto!!!). Suppose that $\varphi = \varphi(\cdot)$ defines a periodic traveling wave solution to \eqref{VBL}, that is, there exists some $c \in \R$ such that $u(x,t) = \varphi(x-ct)$ is a solution to \eqref{VBL} and the profile function $\varphi$ is $L$-periodic and at least of class $C^2$ of its argument. Therefore, $\varphi = \varphi(x)$ belongs to $X = \Hdper([0,L])$. If we denote the unique solution to the Cauchy problem \eqref{CpVBL} with initial datum $\varphi \in X$ as $u_\varphi = \Upsilon(\varphi) \in C([0,T];X)$, then for each $x \in [0,L]$ a.e. there holds $u_\varphi(t)(x) = U_\varphi(t,x) := \varphi(x-ct)$, or, in other words, 
%\[
%u_\varphi(t) = \varphi(\cdot - ct) = \zeta_{-ct}(\varphi) \in X, 
%\]
%where $\zeta_\eta$ is the translation operator in $X = \Hdper$ for any $\eta \in \R$. Indeed, since $\varphi \in C^2$ is a solution to the profile equation \eqref{profileq} then
%\[
%\begin{aligned}
%\partial_t U_\varphi &= -c \varphi'(x-ct) = \varphi''(x-ct) + g(\varphi(x-ct)) - f'(\varphi(x-ct)) \varphi'(x-ct) \\ &=\partial_x^2 U_\varphi + g(U_\varphi) - f'(U_\varphi) \partial_x U_\varphi.
%\end{aligned}
%\]
%Moreover, it is clear that $U_\varphi(0) = \varphi$. By uniqueneess of the solution, we have the conclusion.

Our second observation is the content of the following
\begin{lemma}[global well-posedness of the linearized problem]
\label{lemglobalwp}
Let $f \in C^4(\R)$, $g \in C^3(\R)$. Then for every $\phi \in X_2 = \Hdper([0,L])$ and all $T > 0$ there exists a unique solution $v_\phi \in C([0,T];X_2) \cap C^1([0,T]; \Ldper([0,L]))$ to the Cauchy problem for the linearized operator around the periodic traveling wave $\varphi$, namely,
\begin{equation}
\label{Cplin}
\begin{aligned}
v_t &= \cL_0^c v,\\
v(0) &= \phi,
\end{aligned}
\end{equation}
where
\begin{equation}\label{linear}
\cL_0^c v = v_{xx} + g'(\varphi) v - (f'(\varphi)v)_x + cv_x.
\end{equation}

\end{lemma}
\begin{proof}
Follows similarly as the proof for the nonlinear well-posedness result in periodic Sobolev spaces of section \ref{secwellpos}. The fact that the solution is now global is a consequence of the well-posedness and regularity for parabolic linear problems (see \cite{TayPDE3-2e}). We omit the details.
\end{proof}

\begin{remark}
Recall that $\cL_0^c$ denotes the linearized operator around the wave defined on the periodic Lebesgue space $\Ldper([0,L])$ (with Bloch parameter, or Floquet exponent, $\theta = 0$, see \eqref{Blochop}). The operator is defined in terms of the traveling wave profile $\varphi \in X_2$, its fundamental period $L$ and its speed, $c$.
\end{remark}

Let us now define a mapping which plays the role of the operator $\cS$ in the abstract Corollary \ref{corhenry}. For each $\phi \in X_2 $, set
\begin{equation}
\label{defS}
\begin{aligned}
\cS &: X_2 \to X_2,\\
\cS(\phi) &:= \zeta_{cT} (u_\phi(T))
\end{aligned}
\end{equation}
%\begin{equation}
%\label{defS}
%\cS : X \to X, \qquad \cS(\phi) := \zeta_{cT} (u_\phi(T)),
%\end{equation}
where $u_\phi = \Upsilon(\phi)$ denotes the unique solution to the Cauchy problem \eqref{CpVBL} with $u_\phi(0)=\phi$, $u_\phi \in C([0,{T}];X_2)$. Recall that $u_\phi$ is given by the variation of constants formula \eqref{uphi}.

\begin{lemma}[properties of $\cS$]
\label{propS} 
Let $\varphi$ be a periodic profile for equation \eqref{VBL}. The mapping $\cS$ defined in \eqref{defS} satisfies:
\begin{itemize}
\item[(a)] $\cS(\varphi) = \varphi \in X_2$.
\item[(b)] $\cS$ is twice Fr\'echet differentiable in an open neighborhood of $\varphi$.
\item[(c)] For every $\psi \in X_2$ there holds
\begin{equation}
\label{devS}
\cS'(\varphi) \psi = v_\psi(T),
\end{equation}
where $v_\psi (t) \in X_2$ denotes the unique solution to the linear Cauchy problem \eqref{Cplin} with initial datum $v_\psi (0) = \psi$.
\end{itemize}
\end{lemma}
\begin{proof}
First, notice that $\cS(\varphi) = \zeta_{cT}(u_{\varphi}(T)) = \zeta_{cT} (\zeta_{-cT}(\varphi)) = \varphi$ in view of \eqref{solprofil}. That is, $\varphi \in X_2$ is a fixed point of $\cS$, showing (a). Now, from Theorem \ref{teolocale} we know that the data-solution map $\phi \mapsto \Upsilon(\phi) = u_\phi$ is of class $C^2$. Also, the translation operator is of class $C^2$ in $\Hdper$ ($C^\infty$ indeed). Hence, the composition is of class $C^2$ and we conclude that $\cS$ is twice Fr\'echet differentiable in an open neighborhood, $\Omega = \{ \phi \in X_2 \, : \, \| \phi - \varphi \|_2 < \eta \}$, of $\varphi$. This proves (b).

Therefore, we obtain the Fr\'echet derivative by computing the (G\^ateaux derivative) operator,
\[
\cS'(\varphi) \psi = \frac{d}{d\vep} \big( \cS(\varphi + \vep \psi) \big)_{|\vep = 0},
\]
for any arbitrary $\psi \in X_2$. First observe that, by definition, $\cS(\varphi + \vep \psi) = \zeta_{cT} \big( u_{\varphi + \vep \psi} (T) \big) = \zeta_{cT} \big( \Upsilon(\varphi + \vep \psi) (T) \big)$. Since $\Upsilon$ is of class $C^2$ around $\varphi$ we make the expansion
\begin{equation}
\label{exp1}
u_{\varphi + \vep \psi} = \Upsilon (\varphi + \vep \psi) = \Upsilon(\varphi) + \vep \Upsilon'(\varphi) \psi + O(\vep^2).
\end{equation}
From formula \eqref{uphi} we know that
\[
\begin{aligned}
u_{\varphi + \vep \psi}(t) &= \cV(t) (\varphi + \vep \psi) + \int_0^t \cV(t-\tau) F(u_{\varphi + \vep \psi}, \partial_x u_{\varphi + \vep \psi})(\tau) \, d\tau \\ 
&= \cV(t) (\varphi + \vep \psi) + \int_0^t \cV(t-\tau) \big[ g(u_{\varphi + \vep \psi}) - f'(u_{\varphi + \vep \psi}) \partial_x u_{\varphi + \vep \psi} \big] \, d\tau.
\end{aligned}
\]
Substituting \eqref{exp1} and recalling $\Upsilon(\varphi) = u_\varphi$ we arrive at the expansions,
\[
\begin{aligned}
g(u_{\varphi + \vep \psi}) &= g(u_\varphi) + \vep g'(u_\varphi) \Upsilon'(\varphi) \psi + O(\vep^2),\\
f'(u_{\varphi + \vep \psi}) \partial_x \Upsilon (\varphi + \vep \psi) &= f'(u_\varphi) \partial_x u_\varphi + \vep \partial_x \big( f'(u_\varphi) \Upsilon'(\varphi) \psi \big) + O(\vep^2).
\end{aligned}
\]
Substitution into the previous  integral formula yields
\[
\begin{aligned}
u_{\varphi + \vep \psi}(t) = \cV(t) \varphi + \int_0^t \cV(t-\tau) \big[ g(u_\varphi) - f'(u_\varphi) \partial_x u_\varphi \big] \, d\tau + \vep V_{\varphi,\psi}(t) +  O(\vep^2),
\end{aligned}
\]
where
\[
V_{\varphi,\psi}(t) :=
 \cV(t) \psi + \int_0^t \cV(t-\tau) \big[ g'(u_\varphi) \Upsilon'(\varphi) \psi - \partial_x \big( f'(u_\varphi) \Upsilon'(\varphi) \psi \big) \big] \, d\tau.
\]
Upon differentiation, we notice that
\[
\frac{d}{d\vep} \big( u_{\varphi + \vep \psi}(t) \big)_{|\vep = 0} = \frac{d}{d\vep} \big( \Upsilon(\varphi)(t) + \vep (\Upsilon'(\varphi) \psi)(t) + O(\vep^2) \big)_{|\vep = 0} = (\Upsilon'(\varphi) \psi)(t),
\]
and therefore
\[
V_{\varphi,\psi}(t) = (\Upsilon'(\varphi) \psi)(t) \in X_2,
\]
for all $t \in [0, {T}]$. Then we have shown that $V_{\varphi,\psi}$  is a solution to the integral equation
\begin{equation}
\label{exprV}
V_{\varphi,\psi}(t) = \cV(t) \psi + \int_0^t \cV(t-\tau) \big[ g'(u_\varphi) V_{\varphi,\psi}(\tau) - \partial_x \big( f'(u_\varphi) V_{\varphi,\psi}(\tau) \big) \big] \, d\tau,
\end{equation}
for all $t \in [0, {T}]$. From formula \eqref{exprV} we recognize that $V_{\varphi,\psi}(0) = \psi$ and that it is the solution to a linearized Cauchy problem \eqref{Cplin} with $c=0$. We claim that
\begin{equation}
\label{claimV}
 v_\psi(t):=\zeta_{ct} \big(V_{\varphi,\psi}(t) \big),\quad  t \in [0, {T}],
\end{equation} 
is  the unique solution to the linearized Cauchy problem \eqref{Cplin} with initial datum $\psi$. Indeed, first notice that $\zeta_0 \big(V_{\varphi,\psi}(0) \big) = V_{\varphi,\psi}(0) = \psi$. Now, for $x \in [0,L]$ let us denote
\[
V(t,x) := \zeta_{ct} \big(V_{\varphi,\psi}(t) \big)(x) = V_{\varphi,\psi}(t)(x+ct) = V_{\varphi,\psi}(t, x + ct).
\]
Hence, from \eqref{exprV} and since $\cT = \partial_x^2$ is the infinitesimal generator of the semigroup $\cV(t)$, we obtain
\[
\begin{aligned}
\partial_t V &= \partial_t V_{\varphi,\psi}(t, x +ct) + c \partial_x V_{\varphi,\psi}(t,x+ct) \\
&= \partial_x^2 V_{\varphi,\psi}(t,x+ct) + g'(u_\varphi(x+ct)) V_{\varphi,\psi}(t,x+ct) +\\
&\quad - \partial_x \big( f'(u_{\varphi}(x+ct)) V_{\varphi,\psi}(t,x+ct) ) + c \partial_x V_{\varphi,\psi}(t,x+ct) \\
&= \partial_x^2 V + g'(\varphi(x)) V - \partial_x \big( f'(\varphi(x)) V \big) + c \partial_x V,
\end{aligned}
\]
because $u_\varphi(\cdot + ct) = \varphi(\cdot -ct + ct) = \varphi(\cdot)$. This shows that $\partial_t V = \cL_0^c V$, $V(0) = \psi$ and therefore it is a solution to \eqref{Cplin}. By uniqueness of the solution, we obtain \eqref{claimV} for all $t \in [0, {T}]$.

Finally, evaluating at $t = T$ we have
\[
\begin{aligned}
\cS'(\varphi) \psi &= \frac{d}{d\vep} \Big( \zeta_{cT} (\Upsilon(\varphi) (T)) + \vep \zeta_{cT} ( \Upsilon'(\varphi) \psi (T) ) + O(\vep^2) \Big)_{|\vep = 0} \\
&= \zeta_{cT} \big( \Upsilon'(\varphi) \psi (T) \big)= \zeta_{cT} \big(V_{\varphi,\psi}(T) \big)= v_\psi(T),
\end{aligned}
\]
for any $\psi \in X_2$. This shows (c) and the lemma is proved.
\end{proof}

We are now able to prove the main instability result.

\subsection{Proof of Theorem \ref{mainthem}}
Let us consider the eigenfunction $\Psi \in X_2 = \Hper^2([0,L])$, of the linearized operator $\cL_0^c \, : \, \Ldper([0,L]) \to \Ldper([0,L])$, as the initial condition for the linear Cauchy problem \eqref{Cplin}. By Lemma \ref{lemglobalwp} there exists a unique solution $v_\Psi \in C([0,T];X_2) \cap C^1([0,T];\Ldper([0,L])$ with $v_\Psi(0)=\Psi$. If we define, however, $U(t) = e^{\lambda t} \Psi \in X_2$ for all $t \geq 0$ then, clearly, $U \in C([0,T];X_2) \cap C^1([0,T]; \Ldper([0,L]))$, $U(0) = \Psi$ and 
\[
\partial_t U = \lambda e^{\lambda t} \Psi = e^{\lambda t} \cL_0^c \Psi = \cL_0^c \big( e^{\lambda t} \Psi \big) = \cL_0^c U.
\]
Hence, $U$ is a solution to the Cauchy problem \eqref{Cplin} with $U(0) = \Psi$. By uniqueness of the solution we obtain $U(t) = v_\Psi (t)$ in $X_2$ for all $t \geq 0$. Now, define $\mu := e^{\lambda T}$. This yields
\[
\cS'(\varphi) \Psi = v_\Psi  (T) = U(T) = e^{\lambda T} \Psi = \mu \Psi.
\]
This shows that $\mu \in \sigma(\cS'(\varphi))$ with $|\mu| > 1$ because $\Re \lambda > 0$. Thus, the mapping defined in \eqref{defS} on an open neighborhood of $\varphi$ satisfies the hypotheses of Corollary \ref{corhenry}. 
Therefore, for $\Gamma=\mathcal O_\varphi$ being a $C^1$-curve of fixed points of $\cS$, we conclude that the periodic traveling wave $\varphi$ is orbitally unstable in the space $X_2 = \Hdper([0,L])$. The proof is complete.

\qed

\section{Applications}
\label{secappl}

In order to apply the orbital instability criterion to specific examples, let us write the following hypotheses on the nonlinear functions $f$ and $g$, which were considered by the authors in \cite{AlPl21} in their existence and spectral stability analysis. These hypotheses describe a particular class of viscous balance laws:
%\begin{enumerate}[label=\roman(*)]
%\item Hypothesis 1
%\item Hypothesis 2
%\end{enumerate}

\begin{itemize}
\item[(A$_1$)] $f \in C^4(\R)$. \phantomsection\label{A1} 
\item[(A$_2$)] \phantomsection\label{A2}  $g \in C^3(\R)$ and it is of Fisher-KPP type, satisfying
\begin{equation}
\label{FisherKPP}
\begin{aligned}
	&g(0) = g(1) =0,\\
	&g'(0) > 0, \, g'(1) < 0,\\
	&g(u)>0 \, \textrm{ for all } \, u \in(0,1),\\
	&g(u)<0 \, \textrm{ for all } \, u \in (-\infty,0).
\end{aligned}
\end{equation}
\item[(A$_3$)] \phantomsection\label{A3}  There exists $u_* \in (-\infty, 0)$ such that
\[
\int_{u_*}^0 g(s) \, ds + \int_0^1 g(s) \, ds= 0.
\]
\item[(A$_4$)] \phantomsection\label{A4}  Genericity condition:
\begin{equation}
\label{gencon}
\overline{a}_0 := f'''(0) - \frac{f''(0) g''(0)}{\sqrt{g'(0)}} \neq 0.
\end{equation}
\item[(A$_5$)] \phantomsection\label{A5} Non-degeneracy condition:
\begin{equation}
\label{nondeg}
\left[ \int_{u_*}^1 \!\!\gamma(s) \, ds \right] \left[ \int_{u_*}^1 \!\! f'(s) \sqrt{1+\gamma'(s)^2} \, ds\right] \neq \left[ \int_{u_*}^1 \!\! \sqrt{1+\gamma'(s)^2} \, ds \right] \left[ \int_{u_*}^1 \!\! f'(s) \gamma(s) \, ds \right], 
\end{equation}
where 
\begin{equation}
\label{defPsi}
\gamma(u):= \sqrt{2\int_u^1 g(s) \, ds} \, > 0, \qquad u \in (u_*,1).
\end{equation}
\item[(A$_6$)] \phantomsection\label{A6} Saddle condition:
\begin{equation}
\label{saddle}
f'(1) \left[ \int_{u_*}^1 \gamma(s) \, ds \right]  \neq \int_{u_*}^1 f'(s) \gamma(s) \, ds. %\tag{H$_6$}
\end{equation}
\end{itemize}

%\begin{equation}
%\label{H1}
%\tag{H$_1$}
%f \in C^4(\R).
%\end{equation}
%\begin{equation}
%\label{H2}
%\tag{H$_2$}
%\begin{minipage}[c]{4in}
%$g \in C^3(\R)$ and it is of Fisher-KPP type, satisfying
%\[
%	\begin{aligned}
%	&g(0) = g(1) =0,\\
%	&g'(0) > 0, \, g'(1) < 0,\\
%	&g(u)>0 \, \textrm{ for all } \, u \in(0,1),\\
%	&g(u)<0 \, \textrm{ for all } \, u \in (-\infty,0).
%	\end{aligned}
%\]
%\end{minipage}
%\end{equation}
%\begin{equation}
%\label{H3}
%\tag{H$_3$}
%\begin{minipage}[c]{4in}
%There exists $u_* \in (-\infty, 0)$ such that
%\[
%\int_{u_*}^0 g(s) \, ds + \int_0^1 g(s) \, ds= 0.
%\]
%\end{minipage}
%\end{equation}
%\begin{equation}
%\label{H4}
%\tag{H$_4$}
%\overline{a}_0 := f'''(0) - \frac{f''(0) g''(0)}{\sqrt{g'(0)}} \neq 0, \qquad \text{(genericity condition).}
%\end{equation}
%
%Under hypotheses \eqref{H2} and \eqref{H3}, the value $u_* \in (-\infty,0)$ for which \eqref{H3} holds is unique and, moreover,
%\[
%\int_u^1 g(s) \, ds > 0, \qquad \text{for all } \; u \in (u_*,1).
%\]
%Therefore we can define the function
%\begin{equation}
%\label{defPsi}
%\gamma(u):= \sqrt{2\int_u^1 g(s) \, ds}, \qquad u \in (u_*,1),
%\end{equation}
%and express further hypotheses:
%\begin{equation}
%\label{H5}
%\left[ \int_{u_*}^1 \gamma(s) \, ds \right] \left[ \int_{u_*}^1 f'(s) \sqrt{1+\gamma'(s)^2} \, ds\right] \neq \left[ \int_{u_*}^1 \sqrt{1+\gamma'(s)^2} \, ds \right] \left[ \int_{u_*}^1 f'(s) \gamma(s) \, ds \right], \tag{H$_5$}
%\end{equation}
%\begin{equation}
%\label{H6}
%f'(1) \left[ \int_{u_*}^1 \gamma(s) \, ds \right]  \neq \int_{u_*}^1 f'(s) \gamma(s) \, ds. \tag{H$_6$}
%\end{equation}

\begin{remark}
Although hypotheses \hyperref[A1]{(A$_1$)} - \hyperref[A6]{(A$_6$)} may at first glance seem too restrictive, they are fulfilled by a large number of models including, for example, the well-known Burgers-Fisher equation (cf. \cite{LeHa16,AlPl21,MiGu02}),
\begin{equation}
\label{eqBF}
u_t + uu_x = u_{xx} + u(1-u), 
\end{equation}
for which the nonlinear flux function is given by the paradigmatic Burgers' flux \cite{Bur48,La57}, $f(u) = \tfrac{1}{2} u^2$, and the reaction term is the classical logistic growth function, $g(u) = u(1-u)$ (cf. \cite{Fis37,KPP37}). Other models that satisfy these assumptions include the Buckley-Leverett flux \cite{BuLe42} together with the logistic reaction, yielding the scalar equation
\begin{equation}
\label{LogBLmodel}
u_t + \partial_x \left( \frac{u^2}{u^2 + \tfrac{1}{2}(1-u)^2}\right) = u_{xx} + u(1-u),
%f(u) = \frac{u^2}{u^2 + \tfrac{1}{2}(1-u)^2},
\end{equation}
as well as the modified Burgers-Fisher equation,
\begin{equation}
\label{MBFmodel}
u_t + \partial_x \big( \tfrac{1}{4}u^4 - \tfrac{1}{3} u^3\big) = u_{xx} + u-u^4,
\end{equation}
just to mention a few. See \cite{AlPl21} for more details.
\end{remark}

\begin{remark}
The most important assumption is \hyperref[A2]{(A$_2$)}, which specifies a balance (or production) term with logistic response, that is, with an unstable equilibrium point at $u = 0$ and a stable one at $u=1$. Reaction functions of logistic type are used to model dynamics of populations with limited resources, which saturate into a stable equilibrium point associated to an intrinsic carrying capacity (in this case, the equilibrium state $u = 1$). They are also known as source functions of Fisher-KPP \cite{Fis37,KPP37} or monostable type. 
%
%Hypothesis \hyperref[A1]{(A$_1$)} is a minimal regularity assumption on $f$ to guarantee the existence of small amplitude periodic waves.  Assumption \hyperref[A3]{(A$_3$)} is the balance of forces condition (if we interpret $g$ as the derivative of a potential) necessary for the existence of a homoclinic orbit. Assumption \hyperref[A4]{(A$_4$)} guarantees a (first order, or non-degenerate) Hopf bifurcation from which small-amplitude periodic orbits with bounded fundamental period emerge. Assumptions \hyperref[A5]{(A$_5$)} and \hyperref[A6]{(A$_6$)} make sure that a bifurcation from a limit cycle occurs from a homoclinic loop (traveling pulse type solution) giving rise to bounded periodic waves with large fundamental period.
\end{remark}

As it is established in \cite{AlPl21}, the unstable nature of the origin ($g'(0) > 0$) is responsible of both the existence and spectral instability of small-amplitude periodic waves emerging from a local Hopf bifurcation, as well as the existence and spectral instability of large-period waves which emerge from a global homoclinic bifurcation near a traveling pulse based on a saddle. Even if you change the diffusion mechanism, the unstable equilibrium of the reaction produce similar results (see, for example, the recent paper \cite{AMP22}, where hyperbolic systems with logistic source are considered). It is to be observed that these periodic waves do not exhibit intrinsic symmetries, inasmuch as the existence analysis does not rely on the standard construction techniques but on bifurcation analyses. The existence proof (which is based on local and global bifurcations) also provides the tools to analyze their spectrum. For instance, it can be shown that, for both families of waves, the Floquet spectrum intersects the unstable complex half plane and, hence, they are spectrally unstable (for details see \cite{AlPl21}). 

Let us apply the criterion established in Theorem \ref{mainthem} to both families, yielding their orbital instability in Sobolev periodic spaces with the same period as the fundamental period of the underlying wave. We first examine the case of small-amplitude waves.

\subsection{Orbital instability of small-amplitude periodic waves}

Under assumptions \hyperref[A1]{(A$_1$)} - \hyperref[A4]{(A$_4$)}, the profile equation \eqref{profileq} defines a first order ODE system in the phase plane for which the origin is a center for a critical value of the speed, $c_0$, and where a local Hopf bifurcation occurs when the speed $c$ crosses $c_0$. Then small-amplitude periodic orbits, with period of order $O(1)$, emerge. This behavior can be stated as follows.

\begin{theorem}[existence of small amplitude periodic waves \cite{AlPl21}]
\label{thmexbded}
Suppose that conditions \emph{\hyperref[A1]{(A$_1$)}} - \emph{\hyperref[A4]{(A$_4$)}} hold. Then there exist a critical speed given by $c_0 := f'(0)$ and some $\ep_0 > 0$ sufficiently small such that, for each $0 < \ep < \ep_0$ there exists a unique (up to translations) periodic traveling wave solution to the viscous balance law \eqref{VBL} of the form $u(x,t) = \varphi^\ep(x - c(\ep)t)$, with speed $c(\ep) = c_0 + \ep$ if $\overline{a}_0 > 0$, or $c(\ep) = c_0 - \ep$ if 
$\overline{a}_0 < 0$, and with fundamental period,
\begin{equation}
\label{fundphopf}
L_\ep = \frac{2 \pi}{\sqrt{g'(0)}} + O(\ep), \qquad \text{as } \, \ep \to 0^+.
\end{equation}
The profile function $\varphi^\ep = \varphi^\ep(\cdot)$ is of class $C^3(\R)$, satisfies $\varphi^\ep(x + L_\ep) = \varphi^\ep(x)$ for all $x \in \R$ and is of small amplitude. More precisely,
\begin{equation}
\label{bdsmalla}
|\varphi^\ep(x)|, |(\varphi^\ep)'(x)| \leq C \sqrt{\ep},
\end{equation}
for all $x \in \R$ and some uniform $C > 0$.
\end{theorem}
\begin{proof}
See the proof of Theorem 1.2, \S 2.1, in \cite{AlPl21} for details.
\end{proof}

\begin{remark}
The proof of this existence result is based on a (local) Hopf bifurcation analysis around the critical value $c_0 = f'(0)$ of the wave speed. The bifurcation can be either sub- or supercritical, depending on the sign of $\overline{a}_0$ in \eqref{gencon}.
\end{remark}

\begin{figure}[t]
\begin{center}
\subfigure[$(\varphi,\varphi')$]{\label{figFaseHopfBF}\includegraphics[scale=.465, clip=true]{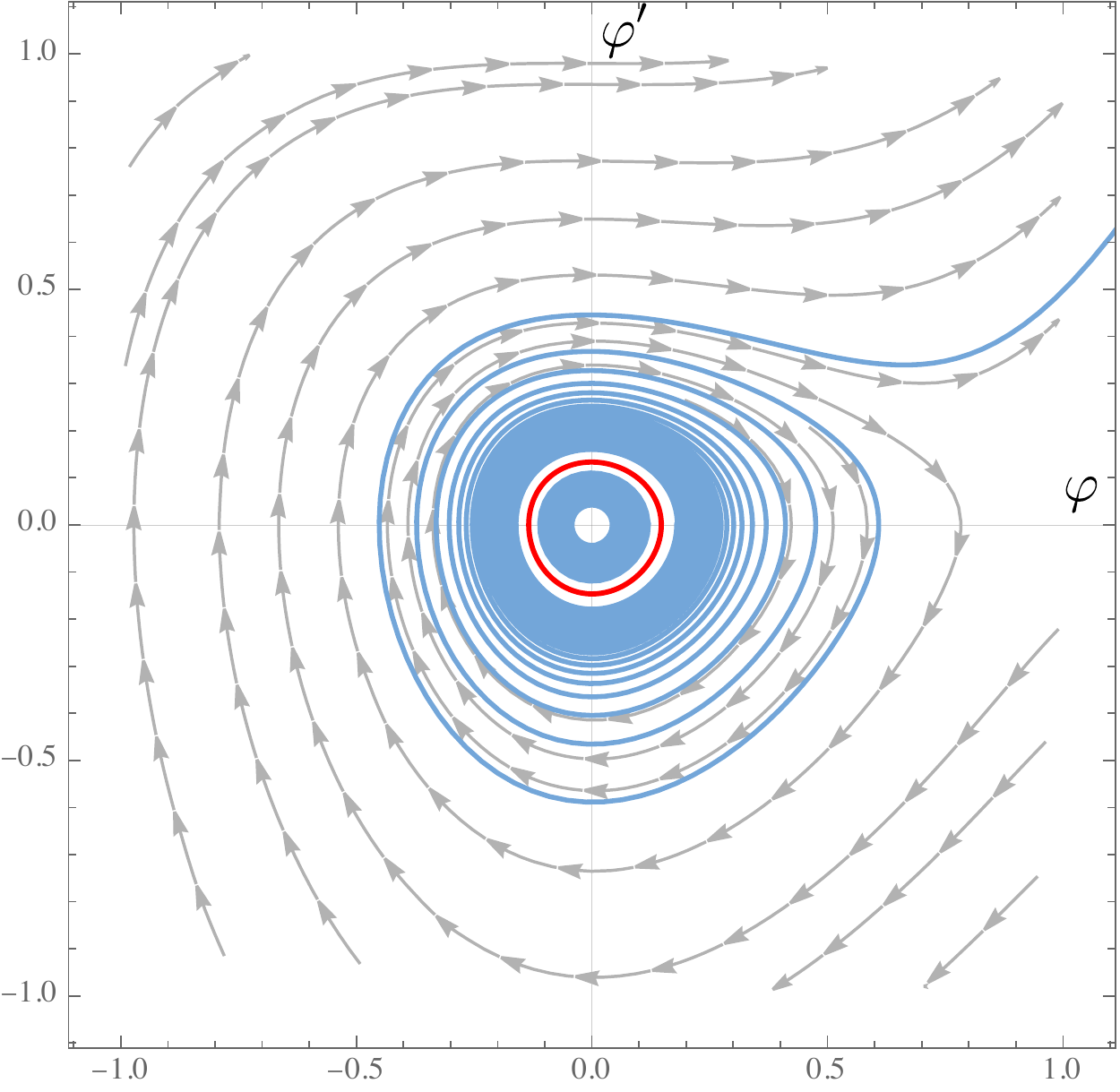}}
\subfigure[$\varphi = \varphi(x)$]{\label{figWaveHopfBF}\includegraphics[scale=.465, clip=true]{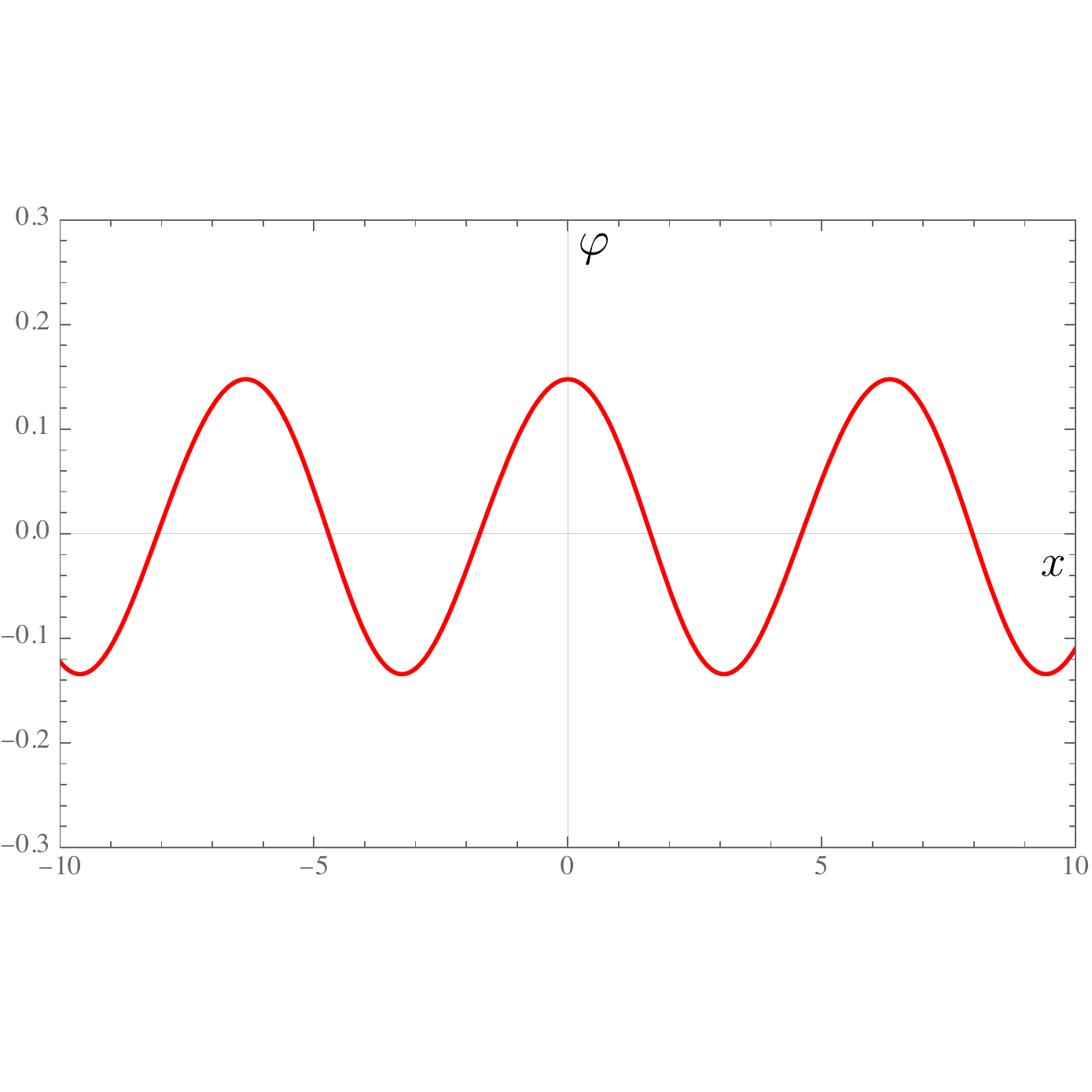}}
%\subfigure[$c= 0.005$]{\label{figHopfBFpos}\includegraphics[scale=.45, clip=true]{HopfBFPos005.pdf}}
%\subfigure[$c= 0.005$]{\label{figHopfBFwave}\includegraphics[scale=.45, clip=true]{HopfBFWave.pdf}}
\end{center}
\caption{\small{Emergence of small-amplitude waves for the Burgers-Fisher equation \eqref{eqBF}. Panel (a) shows the phase portrait (in the $(\varphi,\varphi')$ plane) of the ODE \eqref{profileq} for the speed value $c = 0.005$. Numerical solutions of \eqref{profileq} with nearby initial points are shown in light blue color. The orbit in red is a numerical approximation of the unique small amplitude periodic wave for this speed value. Panel (b) shows the graph (in red) of the approximated periodic wave $\varphi$ as a function of $x$ (color online).}}\label{figHopfBF}
\end{figure}

It is to be noted that formulae \eqref{fundphopf} and \eqref{bdsmalla} imply that, for a fixed small $\ep$, the fundamental period of the wave is of order $O(1)$ and the amplitude of the waves is of order $O(\sqrt{\ep})$, respectively. Thus, one expects that when $\ep \to 0^+$ the small-amplitude waves tend to the origin and the linearized operator (formally) becomes a constant coefficient linearized operator around the zero solution, whose spectrum is determined by a dispersion relation that invades the unstable half plane thanks to the sign of $g'(0)$. This observation is the basis of the analysis in \cite{AlPl21}, which proves that unstable point eigenvalues of the constant coefficient operator split into neighboring curves of Floquet spectra of the underlying small amplitude waves. 

Indeed, the Bloch family of linearized operators around the wave,
\begin{equation}
\label{allBloch}
\left\{
\begin{aligned}
\cL^{c(\ep)}_\theta &:= (\partial_x + i\theta/L_\ep)^2 + a^\ep_1(x) (\partial_x + i \theta/L_\ep) + a^\ep_0(x) \Id,\\
\cL^{c(\ep)}_\theta &: \Ldper([0,L_\ep]) \to \Ldper([0,L_\ep]),
\end{aligned}
\right.
\end{equation}
where
\[
a^\ep_1(x) := c(\ep) - f'(\varphi^\ep), \qquad a^\ep_0(x) := g'(\varphi^\ep) - f'(\varphi^\ep)_x,
\]
for each $\theta \in (-\pi,\pi]$, with domain $\cD(\cL^{c(\ep)}_\theta) = \Hdper([0,L_\ep])$, can be transformed into a family of operators, $\tilde{\cL}^{\ep}_\theta$, defined on the periodic space $\Ldper([0,\pi])$, for which the period no longer depends on $\ep$. 

For that purpose, the authors in \cite{AlPl21} make the change of variables, $y := \pi x/ L_\ep$ and $w(y) := u(L_\ep y/\pi)$, and apply \eqref{fundphopf} and $c(\ep) = c_0 + O(\ep)$, in order to recast the spectral problem for the operators in \eqref{allBloch} as $\widetilde{\cL}^{\ep}_\theta w = \lambda w$, 
where
\[
\begin{aligned}
\widetilde{\cL}^{\ep}_\theta := \widetilde{\cL}^{0}_\theta + \sqrt{\ep} \widetilde{\cL}^{1}_\theta &: \Ldper([0,\pi]) \to \Ldper([0,\pi]),\\
\widetilde{\cL}^{0}_\theta &:= (i \theta + \pi \partial_y)^2 + 4 \pi^2 \Id,\\
\widetilde{\cL}^{1}_\theta &:= b_1(y) (i \theta + \pi \partial_y) + b_0(y) \Id,
\end{aligned}
\]
for each $\theta \in (0,\pi]$ and where the coefficients behave like
\[
\begin{aligned}
b_1(y) &:= \frac{1}{\sqrt{\ep}} a_1^\ep(y) = O(1), \\
b_0(y) &:= \frac{1}{\sqrt{\ep}} \big( a_0^\ep(y) - 4 \pi^2 \big) = O(1),
\end{aligned}
\]
as $\ep \to 0^+$ (for details, see \cite{AlPl21}). It can be shown that, for every $\theta$, $\widetilde{\cL}^{1}_\theta$ is $\widetilde{\cL}^{0}_\theta$-bounded (see Lemma 4.6 in \cite{AlPl21}). Therefore, upon application of standard perturbation theory for linear operators (cf. Kato \cite{Kat80}), it is shown that both spectra, $\sigma(\widetilde{\cL}^{\ep}_\theta)$ and $\sigma(\widetilde{\cL}^{0}_\theta)$, are located nearby in the complex plane for $\ep > 0$ small enough.
%
%
%
%%% End cut
%
%Writing each operator as a perturbation, $\tilde{\cL}^{\ep}_\theta = \tilde{\cL}^{0}_\theta + \sqrt{\ep} \tilde{\cL}^1_\theta$, where $\tilde{\cL}^1_\theta$ is $\tilde{\cL}^0_\theta$-bounded for each $\theta$ (see Lemma 4.6 in \cite{AlPl21}), standard perturbation theory of linear operators (cf. \cite{Kat80}) implies that both spectra, $\sigma(\tilde{\cL}^{\ep}_\theta)$ and $\sigma(\tilde{\cL}^{0}_\theta)$, are nearby in the complex plane. See \cite{AlPl21} for further details.

Transforming back into the original coordinates, the same conclusion holds for any fixed, sufficiently small $\ep > 0$ and the associated family of Bloch operators \eqref{allBloch} defined on $\Ldper([0,L_\ep])$. In particular, for $\theta = 0$, the unperturbed operator
\[
\left\{
\begin{aligned}
\cL_0^{c(0)} &= \partial_x^2 + g'(0) \Id, \\
\cL_0^{c(0)} &:  \Ldper([0,L_\ep])  \to \Ldper([0,L_\ep]),
\end{aligned}
\right.
\]
with $\cD(\cL_0^{c(0)}) = \Hdper([0,L_\ep])$, is clearly self-adjoint with a positive eigenvalue $\widetilde{\lambda}_0 = g'(0)$ associated to the constant eigenfunction $\Psi_0(y) = 1 \in \Hdper([0,L_\ep])$. Hence, the operator $\cL^{c(\ep)}_0$ has discrete eigenvalues $\widetilde{\lambda}_j(\ep)$ in a $\sqrt{\ep}$-neighborhood of $\widetilde{\lambda}_0 = g'(0)$ with multiplicities adding up to the multiplicity of $\widetilde{\lambda}_0$ provided that $\ep$ is sufficiently small. Moreover, since $\widetilde{\lambda}_0 >0$ there holds $\Re \lambda_j(\ep) > 0$. Henceforth, we have the following result.

\begin{lemma}
\label{lemspectcond}
For each $0 < \ep \ll 1$ sufficiently small there holds
\begin{equation}
\ptsp(\cL_0^{c(\ep)})_{|\Ldper} \cap \{ \lambda \in \C \, :  \, | \lambda -  g'(0) |\} \neq \varnothing,
%= \ptsp (\cL_0^0 + \sqrt{\ep} \cL_0^1)_{\Ldper} 
\end{equation}
\end{lemma}
\begin{proof}
See Lemma 4.7 and the proof of Theorem 1.4 in \cite{AlPl21} (in particular, see equation (4.8) in \cite{AlPl21}).
\end{proof}

Therefore, we conclude the existence of an unstable eigenvalue $\lambda(\ep) \in \C$ with $\Re \lambda(\ep) > 0$ and an eigenfunction $\Psi^\ep \in \Hdper([0,L_\ep])$, such that $\cL_0^{c(\ep)} \Psi^\ep = \lambda(\ep) \Psi^\ep$, that is, the spectral instability property holds. Hence, upon application of Theorem \ref{mainthem}, we have the following

\begin{theorem}[orbital instability of small-amplitude periodic waves]
\label{teoorbsmall}
Under assumptions \emph{\hyperref[A1]{(A$_1$)}} - \emph{\hyperref[A4]{(A$_4$)}}, there exists $\bar{\ep}_0 \in (0, \ep_0)$ sufficiently small such that each periodic wave of Theorem \ref{thmexbded}, $u(x,t) = \varphi^\ep(x - c(\ep)t)$, with $\ep \in (0, \bar{\ep}_0)$, is orbitally unstable in the periodic space $X_2 = \Hdper([0,L_\ep])$ under the flow of the viscous balance law \eqref{VBL}.
\end{theorem}

\begin{remark}
It is to be observed that, in the case of small amplitude waves, the instability is due to a structural assumption on the model equations: the instability of the origin as an equilibrium point of the reaction generates an unstable eigenvalue of an associated constant coefficient operator, from which the linearization of a small-amplitude wave represents a perturbation. 
\end{remark}

\subsection{Orbital instability of large-period waves}

The analysis of \cite{AlPl21} also reveals the existence of a different family of periodic waves. Under further assumptions {\hyperref[A5]{(A$_5$)}} and {\hyperref[A6]{(A$_6$)}}, one guarantees that, for another critical value of the speed, the point $(1,0)$ in the phase plane is the (saddle) base of a homoclinic orbit, representing a traveling pulse solution to \eqref{VBL}. Then, from a global bifurcation argument (see, e.g., \cite{SSTC01}) one deduces the existence of large period waves in a vicinity of the homoclinic orbit when the speed tends to the critical value (the speed of the traveling pulse) or, equivalently, when their period goes to infinity. This is the content of the following

\begin{theorem}[existence of large period waves \cite{AlPl21}]
\label{thmexlarge}
Under assumptions \emph{\hyperref[A1]{(A$_1$)}} - \emph{\hyperref[A3]{(A$_3$)}}, \emph{\hyperref[A5]{(A$_5$)}} and \emph{\hyperref[A6]{(A$_6$)}}, there is a critical speed given by
\begin{equation}
\label{defc1}
c_1 := \frac{\int_{u_*}^1 f'(s) \gamma(s) \, ds}{\int_{u_*}^1 \gamma(s) \, ds},
\end{equation}
such that there exists a traveling pulse solution (homoclinic orbit) to equation \eqref{VBL} of the form  $u(x,t) = \varphi^0(x - c_1 t)$, traveling with speed $c_1$ and satisfying $\varphi^0 \in C^3(\R)$, $\varphi^0(x) \to 1$ as $x \to \pm\infty$, with
\begin{equation}
\label{expulse}
|\varphi^0(x) - 1|, |(\varphi^0)'(x)| \leq C e^{-\kappa |x|},
\end{equation}
for all $x \in \R$ and some $\kappa > 0$. In addition, one can find $\ep_1 > 0$ sufficiently small such that, for each $0 < \ep < \ep_1$ there exists a unique periodic traveling wave solution to the viscous balance law \eqref{VBL} of the form $u(x,t) = \varphi^\ep(x - c(\ep)t)$, traveling with speed $c(\ep) = c_1 + \ep$ if $f'(1) < c_1$ or $c(\ep) = c_1 - \ep$ if $f'(1) > c_1$, with fundamental period
\begin{equation}
\label{fundphomo}
L_\ep = O(| \log \ep |) \to \infty, 
\end{equation}
and amplitude
\begin{equation}
\label{bdO1}
|\varphi^\ep(x)|, |(\varphi^\ep)'(x)| = O(1),
\end{equation}
as $\ep \to 0^+$. Moreover, these periodic orbits converge to the homoclinic or traveling pulse solution as $\ep \to 0^+$ and satisfy the bounds (after a suitable reparametrization of $x$),
\begin{equation}
\label{bounds}
\begin{aligned}
\sup_{x \in [-\frac{1}{2}L_\ep, \frac{1}{2}L_\ep]} \left( |\varphi^0(x) - \varphi^\ep(x)| + |(\varphi^0)'(x) - (\varphi^\ep)'(x)  |\right) &\leq C \exp \Big( \!-  \frac{\kappa}{2} L_\ep\Big), \\
| c_1 - c(\ep)| &\leq C \exp \big( \!- \kappa L_\ep \big),
\end{aligned}
\end{equation}
for some uniform $C > 0$, the same $\kappa > 0$ and for all $0 < \ep < \ep_1$.
\end{theorem}
\begin{proof}
See the proof of Theorem 1.3, \S 2.3, in \cite{AlPl21} for details.
\end{proof}

\begin{remark}
The proof of this existence result is based on two components. First, it establishes the existence of a traveling pulse for equation \eqref{VBL}, traveling with speed $c = c_1$ given in \eqref{defc1}. This is a consequence of Melnikov's integral method. Second, upon application of Andronov-Leontovich's theorem in the plane it is shown that there exists a family of periodic waves emerging from the homoclinic orbit; the family is parametrized by $\ep = | c_1 - c(\ep)|$, for which each wave travels with speed $c = c(\ep)$ and converges to the traveling pulse as $\ep \to 0^+$. The fundamental period of the family of periodic waves, $L_\ep$, converges to $\infty$ as $\ep \to 0^+$ at order $O(|\log \ep|)$. As an illustration, Figure \ref{figHomoBL} shows the emergence of large period waves for the logistic Buckley-Leverett equation \eqref{LogBLmodel}. For example, it can be proved that the value of the speed of the homoclinic orbit defined in \eqref{defc1}, from which the periodic loops with large period bifurcate, is $c_1 \approx 0.589097$ (see \cite{AlPl21}). Since $c_1 > f'(1) = 0$, Theorem \ref{thmexlarge} then implies that the family of periodic waves with large period emerge for speed values in a neighborhood above the value $c_1$, that is, for $c \in (0.5891, 0.5891 + \ep)$ with $\ep > 0$ small. Figure \ref{figFaseHomoBL} shows a numerical approximation (in the phase plane) of the homoclinic loop to the ODE \eqref{profileq} with speed $c_1$ (dashed line in blue) and of a large-period wave from the family with speed $c \approx c_1 + 0.025$ (continuous line in red). Figure \ref{figWaveHomoBL} shows numerical approximations of the graph (in red) of the large period wave $\varphi$ as a function of $x$, together with the traveling pulse (dashed, blue line).
\end{remark}

\begin{figure}[t]
\begin{center}
\subfigure[$(\varphi,\varphi')$]{\label{figFaseHomoBL}\includegraphics[scale=.465, clip=true]{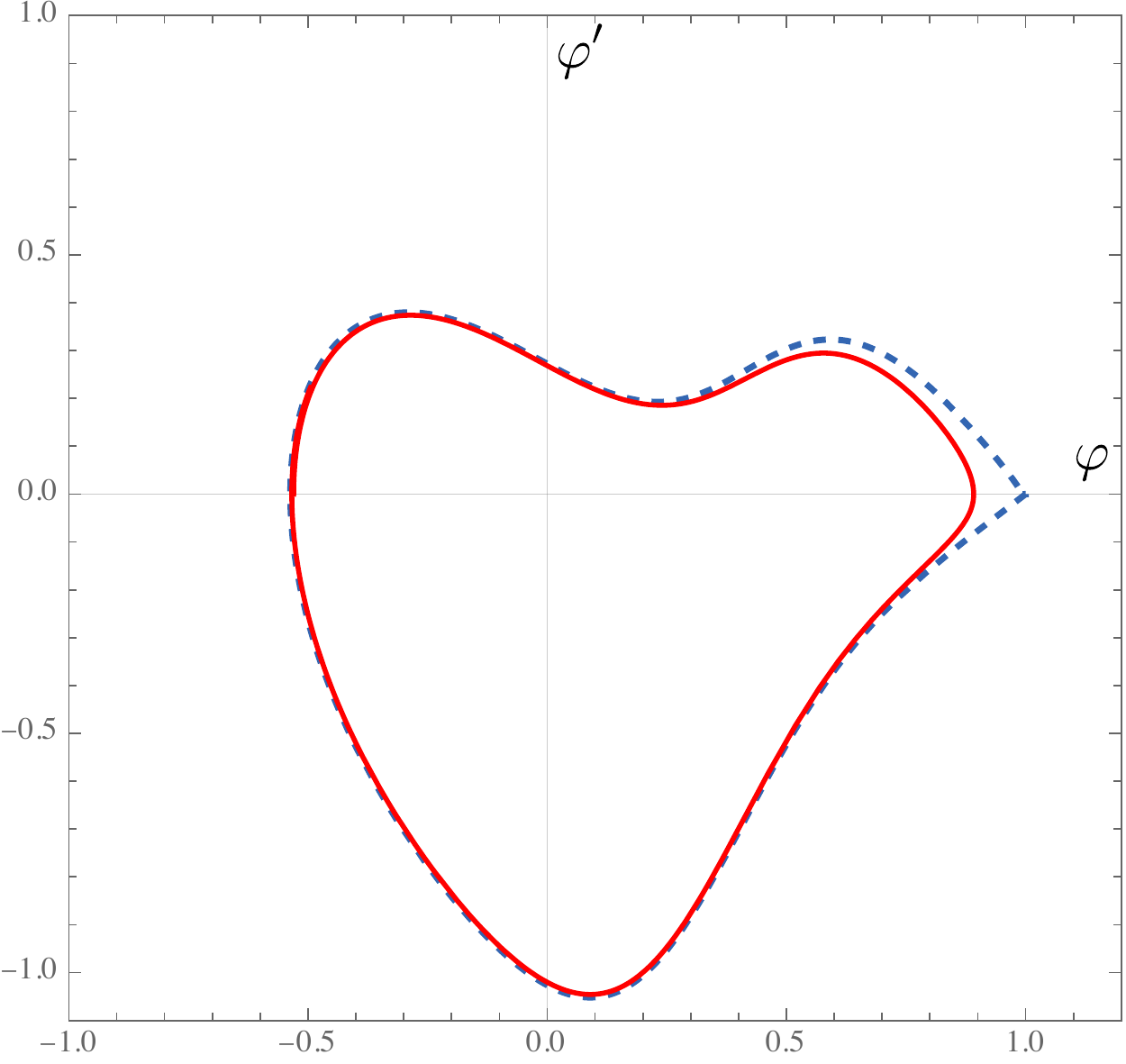}}
\subfigure[$\varphi = \varphi(x)$]{\label{figWaveHomoBL}\includegraphics[scale=.465, clip=true]{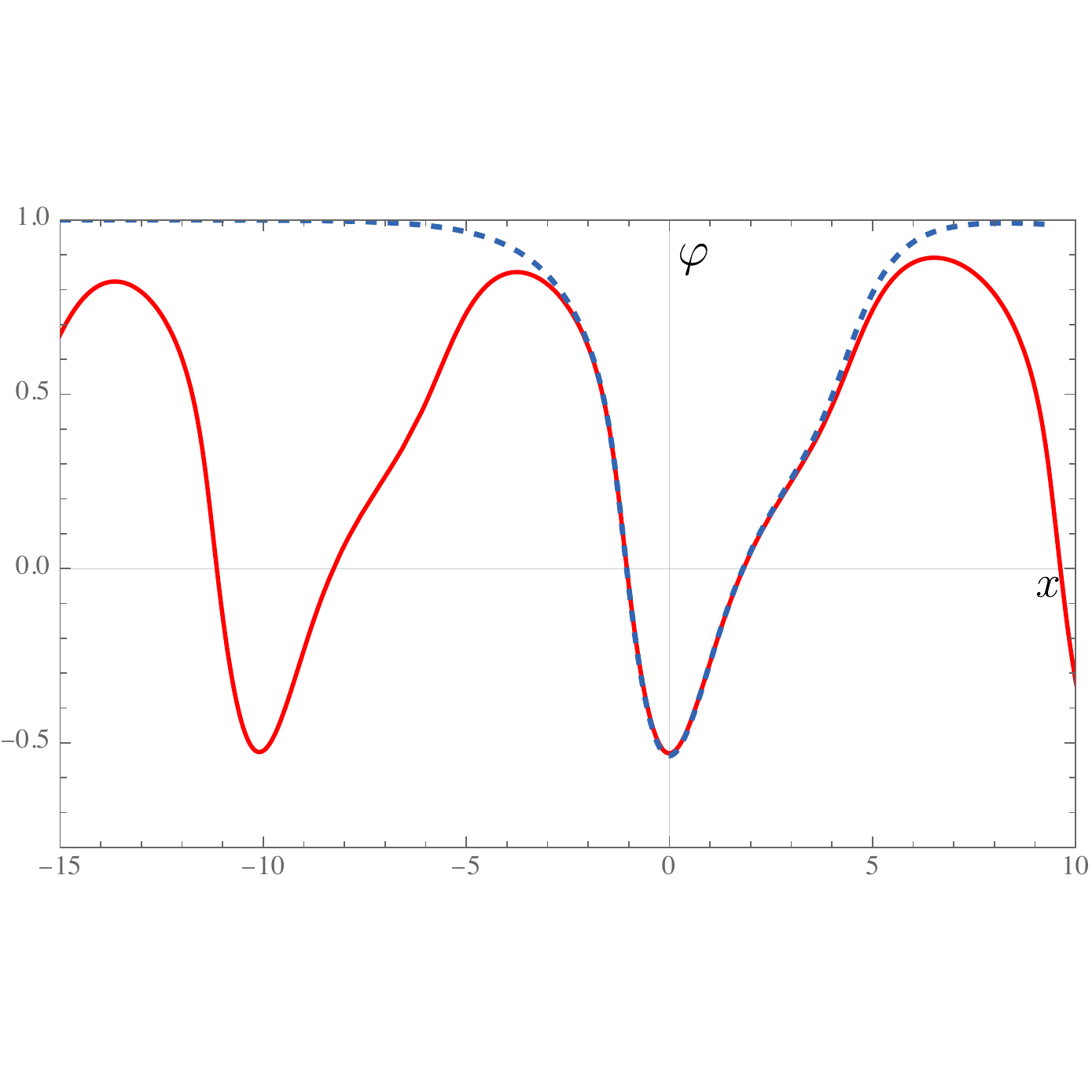}}
%\subfigure[$c= 0.005$]{\label{figHopfBFpos}\includegraphics[scale=.45, clip=true]{HopfBFPos005.pdf}}
%\subfigure[$c= 0.005$]{\label{figHopfBFwave}\includegraphics[scale=.45, clip=true]{HopfBFWave.pdf}}
\end{center}
\caption{\small{Large period waves for the logistic Buckley-Leverett equation \eqref{LogBLmodel}. Panel (a) shows numerical approximations in the phase plane of both the homoclinic loop (traveling pulse) for equation \eqref{LogBLmodel} with speed value $c_1 \approx 0.5891$ (in dashed blue line) and the periodic wave nearby with speed value $c_1  + \epsilon$, $\epsilon \approx 0.025$ (solid, red line). Panel (b) shows numerical approximations of the graph (solid, red line) of the large period wave $\varphi$ as a function of $x$, together with the traveling pulse (dashed, blue line). The period of the wave is of order $O(| \log \epsilon |) \approx O(3.69)$.  (Color online.)}}\label{figHomoBL}
\end{figure}

The proof of existence of this family of waves also underlies the tools to show that their Floquet spectrum is unstable. For instance, if we linearize the equation around the pulse, we obtain the following linear operator:
\begin{equation}
\label{LbarR}
\begin{aligned}
\bar{\cL}^0 &:= \partial_x^2 + \bar{a}_1^0(x) \partial_x + \bar{a}_0^0(x) \Id,\\
\bar{\cL}^0 &: \,  L^2(\R) \longrightarrow L^2(\R).
\end{aligned}
\end{equation}
with smooth coefficients
\[
\begin{aligned}
 \bar{a}_1^0(x) &:= c_1 -  f'(\varphi^0(x)),\\
 \bar{a}_0^0(x) &:= g'(\varphi^0(x)) - f'(\varphi^0(x))_x,
\end{aligned}
\]
which decay exponentially to finite limits as $x \to \pm \infty$; more precisely,
\begin{equation}
\label{expconvpulse}
| \bar{a}_1^0(x) - \bar{a}_1^{\infty} | + | \bar{a}_0^0(x) - \bar{a}_0^{\infty} | \leq C e^{- \kappa |x|}, 
\end{equation}
for all $x \in \R$ with $\bar{a}_1^{\infty} := c_1 - f'(1)$, $\bar{a}_0^{\infty} := g'(1)$. This behavior holds because of the exponential decay of the traveling pulse to hyperbolic end points (see \eqref{expulse}). The operator $\bar{\cL}^0$ is closed and densely defined in $L^2(\R)$ with domain $\cD(\bar{\cL}^0) = H^2(\R)$. Moreover, $\bar{\cL}^0$ is of Sturmian type (see, e.g., Kapitula and Promislow \cite{KaPro13}, \S 2.3) and, upon application of standard Sturm-Liouville theory, we have the following instability result.
\begin{lemma}
\label{theoinspulse}
The traveling pulse solution of Theorem \ref{thmexlarge}, $\varphi^0$, is spectrally unstable; more precisely, there exists $\bar{\lambda}_0 > 0$ such that $\bar{\lambda}_0 \in \ptsp(\bar{\cL}^0)$. Moreover, this eigenvalue is simple.
\end{lemma}
\begin{proof}
See Theorem 5.1 in \cite{AlPl21}.
\end{proof}

The pioneering work by Gardner \cite{Grd97} characterized the spectrum of the linearized operator around a periodic wave of the approximating family and related it to that of the operator $\bar{\cL}^0$. Gardner proved the convergence of both spectra in the infinite period limit and, under very general conditions, that loops of continuous periodic spectra bifurcate from isolated point spectra of the limiting homoclinic wave. Hence, the typical spectral instability of the traveling pulse determines the spectral instability of the periodic waves. Thanks to the convergence estimates \eqref{bounds}, the authors in \cite{AlPl21} verified the hypotheses of a recent refinement of Gardner's result due to Yang and Zumbrun \cite{YngZ19} in order to conclude the spectral instability of the family (see Corollary 4.1 and Proposition 4.2 in \cite{YngZ19}, as well as Theorems 5.2 and 1.5 in \cite{AlPl21}).

In order to apply our orbital instability criterion, however, we need to verify the spectral instability property for the particular Bloch operator with $\theta = 0$. For that purpose, we state the following result which is, in fact, a Corollary of the proof of Theorem 1.5 in \cite{AlPl21}.
\begin{corollary}
\label{corBloch0}
Consider the eigenvalue problem for the Bloch operator \eqref{allBloch} linearized around the family of waves of Theorem \ref{thmexlarge}. Let $\cC \subset \C$ be a positively oriented simple circle of fixed radius with $\cC\subset \{\lambda\in \C: \Re \lambda>0\}$ containing $\bar{\lambda}_0$ (which is the simple, real and unstable isolated eigenvalue of $\bar{\cL}^0$) in its interior, and containing no other eigenvalue of $\bar{\cL}^0$ in the closure of $\cC$. Then for sufficiently small $0 < \ep \ll 1$ and for each $-\pi < \theta \leq \pi$, the Bloch wave spectral problem $\cL^{c(\ep)}_\theta w = \lambda w$ has exactly one point eigenvalue $\lambda = \lambda(\ep)$ in the interior of $\cC$.
\end{corollary}
\begin{proof}
Let us define the matrix coefficients
\[
%\begin{aligned}
\A^0(x,\lambda) := \begin{pmatrix} 0 & 1 \\ \lambda - \bar{a}_0^0(x) & - \bar{a}_1^0(x) \end{pmatrix},
%= \begin{pmatrix} 0 & 1 \\ \lambda - (g'(\varphi^0)-f'(\varphi^0)_z) & - c_1 + f'(\varphi^0) \end{pmatrix}.
%\end{aligned}
\]
for $x \in \R$ and $\lambda \in \C$. These coefficients are clearly analytic in $\lambda$ and of class $C^1(\R;\C^{2 \times 2})$ as functions of $x \in \R$. Moreover, they have asymptotic limits given by
\[
%\begin{aligned}
\A^0_\infty(\lambda) := \lim_{x \to \pm \infty} \A^0(x,\lambda) = \begin{pmatrix} 0 & 1 \\ \lambda - \bar{a}_0^{\infty} & - \bar{a}_1^{\infty} \end{pmatrix} = \begin{pmatrix} 0 & 1 \\ \lambda - g'(1) & - c_1 + f'(1) \end{pmatrix}.
%\end{aligned}
\]
Thanks to exponential decay \eqref{expulse} of the traveling pulse and from continuity of the coefficients, we reckon that, for any $|\lambda| \leq M$ with some $M > 0$, there exists a constant $C(M) > 0$ such that
\begin{equation}
\label{laH2K}
| \A^0(x,\lambda) - \A^0_\infty(\lambda)| \leq C(M) e^{-\kappa |x|},
\end{equation}
for all $x \in \R$. Likewise, define the coefficients,
\[
\A^\ep(x,\lambda) = \begin{pmatrix} 0 & 1 \\ \lambda - \bar{a}_0^{\ep}(x) & - \bar{a}_1^{\ep}(x)\end{pmatrix},
\]
which are analytic in $\lambda \in \C$, continuous in $\ep > 0$ and of class $C^1(\R;\C^{2 \times 2})$ as functions of $x \in \R$. Hence, since the coefficients are smooth and bounded and from estimates \eqref{bounds} we have, for $|\lambda |\leq M$,
\[
\begin{aligned}
|\A^\ep(x, \lambda) - \A^0(x,\lambda) | &\leq \overline{C}(M) \Big( |\varphi^0(x) - \varphi^\ep(x)| + | (\varphi^0)'(x) - (\varphi^\ep)'(x)| + |c_1 - c(\ep)| \Big)\\
&\leq C(M) e^{- \kappa L_\ep/2}.
\end{aligned}
\]
Last estimate, together with \eqref{laH2K} and Theorem \ref{thmexlarge} yields,
\[
\begin{aligned}
%L_\ep = O(|\log \ep|) &\to \infty, \quad \text{as } \, \ep \to 0^+,\\
|\A^0(x,\lambda) - \A^0_{\infty}| &\leq C(M) e^{- \kappa |x|}, \quad \text{for all } \, x \in \R,\\
|\A^0(x,\lambda) - \A^\ep(x,\lambda)| &\leq C(M) e^{-\kappa L_\ep/2}, \quad \text{for all } \, |x| \leq \frac{L_\ep}{2},
\end{aligned}
\]
for every $|\lambda| \leq M$ and some uniform constants $C(M), \kappa > 0$ (see estimates (5.10) in \cite{AlPl21}). We then conclude that the Hypothesis 1 of Theorem 1.2 by Gardner \cite{Grd97} is satisfied (notice that Hypothesis 1 of Gardner requires the estimates for $|\A^0 - \A^\ep|$ in half a period too, because the fundamental period in \cite{Grd97} is $2L_\ep$). Hypothesis 2 is fulfilled in the set of consistent splitting, $\Omega = \{ \lambda \in \C \, : \, \Re \lambda > g'(1) \}$ (see the proof of Theorem 5.1 in \cite{AlPl21}). And Hypothesis 3 is trivially fulfilled by the traveling pulse by Sturm-Liouville theory. Upon application of Theorem 1.2 in \cite{Grd97} and since the eigenvalue $\bar{\lambda}_0 \in \ptsp(\bar{\cL}^0)$ is simple, we conclude the existence of exactly one eigenvalue, $\lambda \in \lambda(\ep)$, of $\cL_\theta^{c(\ep)}$ for each $\theta \in (-\pi,\pi]$, in the interior of $\cC$.
%Consider the family of Bloch operators \eqref{allBloch} around the family of waves of Theorem \ref{thmexlarge}. Then from estimates \eqref{bounds}, \eqref{expulse} and \eqref{expconvpulse} it is possible to prove that
%\begin{equation}
%\label{buena}
%| a_{j}^\ep(x) - \bar{a}_j^0(x) | = O \big( \ep e^{- \kappa |x|} \big), 
%\end{equation}
%for $j = 0,1$ and all $x \in [0,L_\ep]$ as $\ep \to 0^+$. For example, 
%
%Therefore the hypotheses of Theorem 9.7.1, p. 303, in Kapitula and Promislow \cite{KaPro13} hold and the conclusion follows. 
%The lemma is proved
\end{proof}

Henceforth, Corollary \ref{corBloch0} implies that, for the particular case of the Bloch operator with $\theta = 0$, the spectral instability property holds: there exists an unstable eigenvalue $\lambda(\ep) \in \C$ with $\Re \lambda(\ep) > 0$ and an eigenfunction $\Psi^\ep \in \Hdper([0,L_\ep])$, such that $\cL_0^{c(\ep)} \Psi^\ep = \lambda(\ep) \Psi^\ep$. Finally, upon application of the orbital instability criterion (Theorem \ref{mainthem}), we have proved the following
\begin{theorem}[orbital instability of large period waves]
\label{teoorblarge}
Under assumptions \emph{\hyperref[A1]{(A$_1$)}} - \emph{\hyperref[A3]{(A$_3$)}}, \emph{\hyperref[A5]{(A$_5$)}} and \emph{\hyperref[A6]{(A$_6$)}}, there exists $\bar{\ep}_1 \in (0, \ep_1)$ sufficiently small such that each large period wave of Theorem \ref{thmexlarge}, $u(x,t) = \varphi^\ep(x - c(\ep)t)$, with $\ep \in (0, \bar{\ep}_1)$, is orbitally unstable in the periodic space $X_2 = \Hdper([0,L_\ep])$ under the flow of the viscous balance law \eqref{VBL}.
\end{theorem}

\section*{Acknowledgements}

The work of E. \'{A}lvarez was supported by the project CONACYT, FORDECYT-PRONACES 429825/2020 (proyecto aprobado por el CONACYT, FORDECYT-PRONACES 429825/2020), recently renamed as project CF-2019/429825. The work of J. Angulo Pava was partially supported by Grant-CNPq and Universal Project- CAPES/Brazil. J. Angulo would like to express his gratitude to the IIMAS (UNAM), D.F., for its hospitality when this work was carried out. The work of R. G. Plaza was partially supported by DGAPA-UNAM, program PAPIIT, grant IN-104922. 

\def\cprime{$'\!\!$} \def\cprimel{$'\!$}

%% To compile in my computer
%\bibliography{biblio}

%This was the original style for the bibliography
%\bibliographystyle{newstyle}

%the following style is for proofreading
%\bibliographystyle{amsalpha}
%\bibliographystyle{amsplain}

%% 
%% To compile anywhere
%%
% \bibliography{ref}
% \bibliographystyle{newstyle}

\end{document}